\documentclass[11pt,american]{article}
\usepackage{url}            
\usepackage{booktabs}       
\usepackage{amsfonts}       
\usepackage{nicefrac}       
\usepackage{microtype}      
\usepackage{stmaryrd}
\usepackage{multirow}
\usepackage{setspace}
\usepackage{xr}
\usepackage{pdfpages}
\usepackage{mathptmx}

\usepackage[T1]{fontenc}
\usepackage[latin9]{inputenc}
\usepackage[executivepaper]{geometry}
\usepackage{fancyhdr}
\usepackage{amsmath}
\usepackage{amssymb}
\usepackage{stackrel}
\usepackage{setspace}
\usepackage{hyperref}
\usepackage{amsthm}
\usepackage{titlesec}
\usepackage{hyphenat}
\usepackage{float}
\usepackage{xcolor}

\usepackage{makecell}
\usepackage{adjustbox,lipsum}

\makeatletter

\newcommand{\lyxmathsym}[1]{\ifmmode\begingroup\def\b@ld{bold}
  \text{\ifx\math@version\b@ld\bfseries\fi#1}\endgroup\else#1\fi}

\usepackage{ytableau}
\usepackage[enableskew]{youngtab}
\DeclareMathAlphabet{\altmathcal}{OMS}{cmsy}{m}{n}
\usepackage{mathptmx}
\usepackage{array}
\usepackage{graphicx}
\newtheorem{proposition}{Proposition}[section]
\newtheorem{example}[proposition]{Example}
\newtheorem{conjecture}[proposition]{Conjecture}
\newtheorem{theorem}[proposition]{Theorem}
\newtheorem{corollary}[proposition]{Corollary}
\newtheorem{lemma}[proposition]{Lemma}
\newtheorem{claim}[proposition]{Claim}
\theoremstyle{definition}
\newtheorem{definition}[proposition]{Definition}
\theoremstyle{definition}
\newtheorem{remark}[proposition]{Remark}

\newtheorem{observation}[proposition]{Observation}
\newtheorem{problem}[proposition]{Problem}

\theoremstyle{plain}
\newtheorem*{theorem*}{Theorem}
\newtheorem*{lemma*}{Lemma}
\usepackage{glossaries}

\makeatother
\hypersetup{
  colorlinks,
  citecolor=violet,
  linkcolor=blue,
  urlcolor=blue}

\begin{document}

\fancyhead{}
\fancyhead[C]{\textbf{}}

\fancyhead[L]{\nouppercase{\leftmark}}
\fancyhead[R]{\nouppercase{\rightmark}}

\author{Alon Romano}
\date{ }
\title{\onehalfspacing{}On the $T$-ideal generated by the identity $f=x^{n}$}
\maketitle

\begin{abstract}
 
\noindent Let $\mathbb{F}\left\langle X\right\rangle $ denote the free non-unitary associative algebra generated by $X=\left\{ x_{1},x_{2},...\right\} $
over a field $\mathbb{F}$ of characteristic zero, and let $V_{m}=V_{m}\left(x_{1},..,x_{m}\right)\leq\mathbb{F}\left\langle X\right\rangle $ be the vector space of all multilinear polynomials of degree $m$.
For any natural numbers $n\leq m$, we let $W_{n,m}$ denote the subspace $V_{m}\cap\left(x^{n}\right)^{T}$,
where $\left(x^{n}\right)^{T}$ is the $T$-ideal generated by the
polynomial $x^{n}$. It is well known that $W_{n,m}$ is an $S_{m}$-representation,
so one can decompose 
\begin{gather*}
W_{n,m}\cong\underset{\lambda\vdash m}{\bigoplus}m^{\lambda}S^{\lambda},
\end{gather*}
where $S^{\lambda}$ is the irreducible representation of $S_{m}$
corresponding to the partition $\lambda$. Given a partition $\lambda=\left(\lambda_{1},\lambda_{2},...,\lambda_{r}\right)\vdash m$, we define its length by $\ell\left(\lambda\right)=r$. In addition, for every $\ensuremath{s\in\mathbb{Z}_{\geq0}}$, we set $\lambda^{\left(s\right)}=\left(\lambda_{1}+s,\lambda_{2},...,\lambda_{r}\right)\vdash m+s$. 

\nohyphens{We show that if $K\in\mathbb{N}=\left\{ 1,2,...\right\} $ is fixed, the multiplicity of
$S^{\lambda}$ in $W_{n,n+K}$ equals zero whenever $\ell\left(\lambda\right)>2K+1$. We also show that the multiplicities
of $W_{n,n+K}$ must stabilize, in the following sense: there exists
$N\in\mathbb{N}$ such that, if $W_{N,N+K}$ decomposes as $\underset{\lambda\vdash N+K}{\bigoplus}m^{\lambda}S^{\lambda}$,
then for every $n\geq N$ we have} 
\begin{gather*}
W_{n,n+K}\protect\cong\underset{\lambda\vdash N+K}{\bigoplus}m^{\lambda}S^{\lambda^{\left(n-N\right)}}.
\end{gather*}
The proof uses the
close connection between $S_{m}$-representations and $GL_{k}\left(\mathbb{F}\right)$-representations.
Some general estimates for $\text{dim}\left(W_{n,n+K}\right)$ are
also suggested.
\end{abstract}
 
\tableofcontents

\titlelabel{\thetitle.\quad}
\section{ Main Definitions and Motivation }

Throughout this paper, $\mathbb{F}$ stands for a field of characteristic
zero, and all vector spaces and algebras are taken over $\mathbb{F}$. In addition, we always assume that our algebras are non-unitary. For
a background on PI-algebras and details of the results discussed in
this section, we refer to \cite{drensky2012polynomial} or \cite{giambruno2005polynomial}.

A polynomial identity of an associative algebra $A$ is a polynomial
$f\left(x_{1},...,x_{n}\right)$ in the free associative algebra $\mathbb{F}\left\langle X\right\rangle $
on a countable set $X=\left\{ x_{1},x_{2},...\right\} $ of (non-commutative) variables, such that
$f\left(a_{1},...,a_{n}\right)=0$ for all $a_{1},...,a_{n}\in A$.
In this case, we say that $A$ satisfies the identity $f$. 
Sometimes we will use other symbols (e.g.,  $x,y,z,$ etc.) for the elements of $X$.

A polynomial identity $g\left(x_{1},...,x_{m}\right)\in\mathbb{F}\left\langle X\right\rangle $
is called a \textbf{consequence}  of the polynomial identity $h\left(x_{1},...,x_{k}\right)\in\mathbb{F}\left\langle X\right\rangle $,
if any algebra satisfying the identity $h\left(x_{1},...,x_{k}\right)=0$ 
satisfies the identity $g\left(x_{1},...,x_{m}\right)=0$.

\begin{definition}
A $T$-ideal is an ideal $I\triangleleft\mathbb{F}\left\langle X\right\rangle $
such that $\psi\left(I\right)\subset I$ for all homomorphisms $\psi:\mathbb{F}\left\langle X\right\rangle \rightarrow\mathbb{F}\left\langle X\right\rangle$.
Given $f\in\mathbb{F}\left\langle X\right\rangle $, we denote by
$\left(f\right)^{T}$ the smallest $T$-ideal containing $f$. 

\end{definition}

It is well known  (see, e.g., Prop. 1.2.8 in \cite{drensky2012polynomial}) that in characteristic zero, every $T$-ideal $I\triangleleft\mathbb{F}\left\langle X\right\rangle $ is completely determined by its multilinear elements. Hence, denoting the symmetric group on $[m]=\left\{ 1,...,m\right\} $ by $S_{m}$, and define
\begin{gather*}
V_{m}:=\text{span}_{\mathbb{F}}\left\{ x_{\sigma\left(1\right)}\cdot\cdot\cdot x_{\sigma\left(m\right)}:\sigma\in S_{m}\right\} \subseteq\mathbb{F}\left\langle X\right\rangle ,
\end{gather*}

the structure of $I$ is completely determined by $\left\{ V_{m}\cap I\right\} _{m\in\mathbb{N}}$. 

Our goal in this paper is to gain a better understanding of the subspace
\begin{gather*}
W_{n,m}:=V_{m}\cap\left(x^{n}\right)^{T}\leq V_{m}.
\end{gather*}

\begin{example}
\label{once_ag}
Consider $V_{2}=\text{span}_{\mathbb{F}}\left\{ x_{1}x_{2},x_{2}x_{1}\right\} $.
We claim that $W_{2,2}=\text{span}_{\mathbb{F}}\left\{ x_{1}x_{2}+x_{2}x_{1}\right\} $.
Indeed, since 
\begin{gather*}
x_{1}x_{2}+x_{2}x_{1}=\left(x_{1}+x_{2}\right)^{2}-x_{1}^{2}-x_{2}^{2}\in\left(x^{2}\right)^{T},
\end{gather*}

we have $W_{2,2}\supset\text{span}_{\mathbb{F}}\left\{ x_{1}x_{2}+x_{2}x_{1}\right\} $.
On the other hand, the identity $f=x^{2}$ does not imply the identity
$g=xy$ (hence $x_{2}x_{1},x_{1}x_{2}\notin\left(x^{2}\right)^{T}$):
take any three-dimensional vector space $A$ with a basis $\left\{ a,b,c\right\} $,
and give to $A$ a structure of algebra by setting 
\begin{flalign*}
a\circ a & =b\circ b=c\circ c=0\\
a\circ b & =c=-b\circ a\\
a\circ c & =0=c\circ a\\
b\circ c & =0=c\circ b.
\end{flalign*}

We see that $A$ (which is actually a Lie algebra) satisfies the identity $x^{2}=0$, but not the identity
$xy=0$.

\end{example}

\subsection{Applying Representation Theory}

In a series of papers in the 1970s, Regev made a remarkable connection
between the theory of $S_{m}$-representations and the theory of polynomial
identities (in characteristic zero). We will follow this approach. Note that  $S_{m}$ is acting on $V_{m}$ from the left by 
\begin{gather*}
\sigma.\left(x_{i_{1}}\cdot\cdot\cdot x_{i_{m}}\right)=x_{\sigma\left(i_{1}\right)}\cdot\cdot\cdot x_{\sigma\left(i_{m}\right)}
\end{gather*}
and from the right by 
\begin{gather*}
\left(x_{i_{1}}\cdot\cdot\cdot x_{i_{m}}\right)\sigma=x_{i_{\sigma\left(1\right)}}\cdot\cdot\cdot x_{i_{\sigma\left(m\right)}}.
\end{gather*}
 Define a map $\varphi:\mathbb{F}\left[S_{m}\right]\rightarrow V_{m}$ by setting
\begin{gather*}
\varphi:\underset{\sigma\in S_{m}}{\sum}\alpha_{\sigma}\cdot\sigma\mapsto\underset{\sigma\in S_{m}}{\sum}\alpha_{\sigma}\cdot x_{\sigma\left(1\right)}\cdot\cdot\cdot x_{\sigma\left(m\right)}.
\end{gather*}

It is clear that $\varphi$ is a linear isomorphism, and moreover,
since for any $\sigma,\tau\in S_{m}$
\begin{align}
  \sigma\left(x_{\tau\left(1\right)}\cdot\cdot\cdot x_{\tau\left(m\right)}\right)=x_{\sigma\tau\left(1\right)}\cdot\cdot\cdot x_{\sigma\tau\left(m\right)}=\left(x_{\sigma\left(1\right)}\cdot\cdot\cdot x_{\sigma\left(m\right)}\right)\tau, \label{left_&_right}
\end{align}
$V_{m}$ is isomorphic to $\mathbb{F}\left[S_{m}\right]$ as an $\mathbb{F}\left[S_{m}\right]$-bimodule.
Consequently, we can identify elements in $\mathbb{F}\left[S_{m}\right]$
with their images in  $V_{m}$.

A key observation is that the left action of $S_{m}$ on $V_{m}$
satisfies 
\begin{gather*}
\sigma.f\left(x_{1},...,x_{m}\right)=f\left(x_{\sigma\left(1\right)},...,x_{\sigma\left(m\right)}\right)
\end{gather*}
for any $f\in V_{m}$, which implies that for any $T$-ideal $I\triangleleft\mathbb{F}\left\langle X\right\rangle $,
$V_{m}\cap I$ becomes a left submodule of $\mathbb{F}\left[S_{m}\right]$ (since by definition
$I$ is invariant under substitutions). In particular, 
\begin{gather*}
W_{n,m}=V_{m}\cap\left(x^{n}\right)^{T}
\end{gather*}
 is a left submodule of $\mathbb{F}\left[S_{m}\right]$.
\nohyphens{It is well known that any irreducible representation of $S_{m}$ corresponds
to a partition $\lambda\vdash m$, so we denote these representations
by $\left\{ S^{\lambda}\mid\lambda\vdash m\right\} $. The multiplicity
of $S^{\lambda}$ in $W_{n,m}$ is denoted by $m_{n,m}^{\lambda}$,
(or just by $m^{\lambda}$, when $n$ and $m$ are clear from the context). In this notation, we have
an isomorphism of $S_{m}$-representations}
\begin{align}
  W_{n,m}\cong\underset{\lambda\vdash m}{\bigoplus}m_{n,m}^{\lambda}S^{\lambda}. \label{deco}
\end{align}

\begin{definition}
\label{derived}
Let $\lambda=\left(\lambda_{1},\lambda_{2},...,\lambda_{r}\right)$
be a partition of $N$. For any $n\geq N$ and $\mu\vdash n$, we say that
$\mu$ is \textbf{derived} from $\lambda$, and write $\mu=\lambda^{\left(n-N\right)}$,
if 
\begin{gather*}
\mu=\left(\lambda_{1}+\left(n-N\right),\lambda_{2},...,\lambda_{r}\right).
\end{gather*}
\end{definition}
This definition has a simple visualization: denoting the Young diagram
associated to $\lambda$ by $D_{\lambda}$ (see Definition \ref{Partition} below), then $\mu\vdash n$ is derived from $\lambda\vdash N$
if and only if $D_{\mu}$ is obtained from $D_{\lambda}$ by adjoining
$n-N$ boxes to the right upper corner of $D_{\lambda}$.

\begin{figure}[h]
    \centering
    \includegraphics[width=0.5\textwidth]{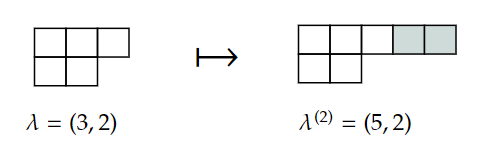}
    \caption{The partition $\mu=\left(5,2\right)$ is derived from $\lambda=\left(3,2\right)$.}
    \label{fig:mesh0}
\end{figure}

\begin{definition}
\label{derived_def}
Suppose that an $S_{N}$-representation $W$ decomposes as $W\cong\underset{\lambda\vdash N}{\bigoplus}m^{\lambda}S^{\lambda}$,
and let $n\geq N$. We say that an $S_{n}$-representation $U$ is \textbf{derived} from $W$, if $U$ decomposes as 
\begin{gather*}
U\cong\underset{\lambda\vdash N}{\bigoplus}m^{\lambda}S^{\lambda^{\left(n-N\right)}}.
\end{gather*}
In other words, for every $\mu\vdash n$, the multiplicity of $S^{\mu}$
in $U$ equals zero, unless $\mu$ is derived (in the sense of Definition \ref{derived}) from some $\lambda\vdash N$,
in which case it equals the multiplicity of $S^{\lambda}$ in $W$.
\end{definition}
\subsection{Main result}
Our main theorem in this paper is the following one:

\begin{theorem}
Fix $K\in\mathbb{N}$. There exists $N\in\mathbb{N}$ such that
for every $n\geq N$, $W_{n,n+K}$
is derived from $W_{N,N+K}$.
\end{theorem}

We will prove this theorem by using the following theorems.

\begin{theorem}
Fix $K\in\mathbb{N}$, and keep the notations from (\ref{deco}).

\textbf{(1)} For every $\lambda=\left(\lambda_{1},...,\lambda_{r}\right)\vdash n+K$
such that $r>2K+1$, $m_{n,n+K}^{\lambda}=0$.

\textbf{(2)} There exists an integer $D$ (depending only on $K$), such
that for every $n$
\begin{gather*}
\underset{\lambda\vdash n+K}{\sum}m_{n,n+K}^{\lambda}\leq D.
\end{gather*} 
\end{theorem}

\begin{theorem}
\textit{Fix $K\in\mathbb{N}$. There exists an integer $M_{K}>0$ such that,
for every $N\geq M_{K}$ and \linebreak $\lambda\vdash N+K$, if $n\geq N$ is
large enough, the multiplicity of $S^{\lambda^{\left(n-N\right)}}$
in $W_{n,n+K}$ is at least the multiplicity of $S^{\lambda}$ in
$W_{N,N+K}$.}
\end{theorem}

We believe, however, that this theorem can be considerably strengthened.

\textbf{Conjecture.}
\textit{In the theorem above, one can take  $M_{K}=1$  and any $n\geq N$.}

\subsection{The motivation for studying \texorpdfstring{$W_{n,m}$}{Lg} }
\nohyphens{Our original motivation for studying the space $W_{n,m}$ was to analyze
the class of nilpotency of algebras satisfying the identity $x^{n}=0$. Let $n$ be a natural number. The classical Nagata-Higman theorem \cite{higman1956conjecture} \cite{10.2969/jmsj/00430296},
which was first proved by Dubnov and Ivanov \cite{dubnov1943abaissement} in 1943, states
that if $A$ is an associative (non-unitary) algebra, and $A$ satisfies
the identity $f=x^{n}$, there exists an integer $m=m\left(n\right)$
such that $A^{m}=0$. In other words, the  identity
$x_{1}\cdot\cdot\cdot x_{m}=0$ is a consequence of the
identity $x^{n}=0$. Consider}  
\begin{gather*}
d\left(n\right):=\text{min}\left\{ m:x_{1}\cdot\cdot\cdot x_{m}=0\text{ is a consequence of  \ensuremath{x^{n}}}=0\right\}.
\end{gather*}

In 1974, Razmyslov \cite{razmyslov1974trace} established the upper bound $d\left(n\right)\leq n^{2}$,
which to this day remains the best known upper bound. For a lower bound, Kuzmin \cite{kuzmin1975nagata} showed that $\frac{n\left(n+1\right)}{2}\leq d\left(n\right)$,
and conjectured that we have $d\left(n\right)=\frac{n\left(n+1\right)}{2}$
for every $n$. As of 2022, this was confirmed only for $n=1,2,3$
by Dubnov and Ivanov \cite{dubnov1943abaissement}, and for $n=4$ by Vaughan-Lee \cite{vaughan1993algorithm}.
Let us reformulate Kuzmin's conjecture in terms of $\text{dim}\left(W_{n,m}\right)$.

\begin{proposition}
\label{lst}
Given $n\in\mathbb{N}$, we have $d\left(n\right)=\text{min}\left\{ m\in\mathbb{N}\mid\text{dim}\left(W_{n,m}\right)=m!\right\} $.
In particular, Kuzmin's conjecture holds if and only if $\text{dim}\left(W_{n,n\left(n+1\right)/2}\right)=\left(n\left(n+1\right)/2\right)!$.
\end{proposition}
\begin{proof}
By definition, $d\left(n\right)\leq m$ if and only if
 $g_{\sigma}=x_{\sigma\left(1\right)}\cdot\cdot\cdot x_{\sigma\left(m\right)}$
is a consequence of $f=x^{n}$ for any $\sigma\in S_{m}$, which holds if and only if $g_{\sigma}$
lies in $V_{m}\cap\left(x^{n}\right)^{T}=W_{n,m}$. As $\left\{ x_{\sigma\left(1\right)}\cdot\cdot\cdot x_{\sigma\left(m\right)}\right\} _{\sigma\in S_{m}}$
is a linear basis of $V_{m}$, we have $\text{dim}\left(V_{m}\right)=m!=\text{dim}\left(W_{n,m}\right)$
if and only if $V_{m}=W_{n,m}$, if and only if $g_{\sigma}\in W_{n,m}$
for every $\sigma\in S_{m}$. 
\end{proof}

In light of this proposition, the investigation of $W_{n,m}$ should focus only on the case  $m\leq n^{2}$, since for every $m>n^{2}$
we have $W_{n,m}=V_{m}\cong\mathbb{F}\left[S_{m}\right]$ (by Razmyslov's upper bound),
and the multiplicities of $W_{n,m}$ are well understood in this case.

\subsection{Outline}

This paper is organized as follows.
In Section 2 we provide the needed background from the representation
theory of $S_{m}$ and $GL_{k}\left(\mathbb{F}\right)$, and conclude
some results from the theory of symmetric tensors. This section has been heavily influenced by Drensky and Benanti's paper \cite{benanti1999polynomial}. 
In Section 3 we make some reductions which, together with the results
from Section 2, lead to the main result (Theorem \ref{stab}). These reductions
are proved in section 4. Finally, in Section 5, we give some upper
and lower bounds on $\text{dim}\left(W_{n,m}\right)$, and suggest a conjecture about
the asymptotic behavior of $\text{dim}\left(W_{n,n+K}\right)$ for
a fixed $K$.

Table \ref{tab:label_test} gives a summary of the notations introduced in this paper.

\subsection{Acknowledgments}

This work is part of my master's thesis. I would like to thank my advisor, Prof. Aner Shalev, for his guidance and encouragement throughout this work.

I would also like to thank Niv Levhari for insightful discussions,
and for his assistance in developing a program that greatly eased the calculations. Special thanks also go to the anonymous referee, whose insightful comments immeasurably enriched this paper.

  \begin{table}[h!]
  \centering
  \begin{adjustbox}{max width=\textwidth}
  \begin{tabular}{*{3}{|c}|}
  \hline
  section introduced & notation & description\\
  \hline
  1,2 & $\lambda^{\left(d\right)}$ & \makecell{The partition obtained from $D_{\lambda}$ by adjoining \\ $d$ boxes 
to the right upper corner of $D_{\lambda}$} \\
  \hline
  1,2 & $S^{\lambda}$ & The irreducible representation of $S_{m}$ corresponding to $\lambda$\\
  \hline
  2 & $V^{\lambda}$ & The Weyl module corresponding to $\lambda$\\
  \hline
  2 & \makecell{$\text{Sub}_{\left(\text{x,y}\right)}$\\$\text{Sub}_{\left(\text{x,y}\right)}^{T}$} & \makecell{The substitution $x_{i}\protect\mapsto y_{i}$ \\ The substitution
according to a tableau $T$}\\
  \hline
  2 & $A_{k}$ & $\mathbb{F}\left\langle x_{1},...,x_{k}\right\rangle $, the free
algebra of rank $k$\\
  \hline
  2 & $\ensuremath{U^{\otimes_{s}n}}$ & The $n$-th symmetric power of $U$\\
  \hline
  2 & $\left[A_{k}\right]^{\left(m\right)}$ & The space of all homogeneous polynomials of degree $m$ in $A_{k}$\\
  \hline
  2 & $\left[U^{\otimes_{s}n}\right]^{\left(m\right)}$ & The space of all the polynomials of degree $m$ in  $U^{\otimes_{s}n}$\\
  \hline
  3 & $P_{n}\left(x_{1},...,x_{n}\right)$ & The symmetric polynomial $\protect\underset{\sigma\in S_{n}}{\sum}x_{\sigma\left(1\right)}\cdot\cdot\cdot x_{\sigma\left(n\right)}$\\
  \hline
3 & $\Omega_{n,m}$ (resp. $\Omega_{n,m}^{\lambda}$) & \makecell{The set of ordered partitions (resp. of type $\lambda$) \\ of $m$, divided into
$n$ parts}\\
  \hline
  4 & $P_{O}$ & The symmetric polynomial corresponding to $O\in\Omega_{n,m}$\\
  \hline
  4 & $T^{\left(s\right)}$ & \makecell{Given $\lambda\vdash m$  and a tableau  $T$  of shape $\lambda$, $T^{\left(s\right)}$
is obtained by \\ adjoining the boxes  $  {\scriptsize \   \begin{ytableau} \scriptstyle m+1 & ... & \scriptstyle m+s  \end{ytableau}  }  $
to the right upper corner of $T$}\\
  \hline
  4 & $P^{\left(s\right)}$ & \makecell{If $P=P_{n}\left(y_{1},...,y_{n}\right)$, then $P^{\left(s\right)}=P_{n+s}(y_{1},...,y_{n},\stackrel{s\text{ times}}{\protect\overbrace{y_{1},...,y_{1}}})$}\\
  \hline
  4 & $\text{C}\left(u\right)$ & \makecell{If $u$ is a monomial, then $\text{C}\left(u\right)$ is  the \\ largest
block of the form $y_{1}^{d}$ inside $u$}\\
  \hline
  4 & $u^{\left(s\right)}$ & \makecell{If $u$ is a monomial of the form $u=a\cdot\text{C}\left(u\right)\cdot b$,\\
then $u^{\left(s\right)}=a\cdot y_{1}^{s}\text{C}\left(u\right)\cdot b$}\\
  \hline
\end{tabular}
\end{adjustbox}
\caption{Summary of Notations}
\label{tab:label_test}
\end{table}

\section{Preliminaries and first results}

\nohyphens{In this section we review some parts of the classical representation
theory of $S_{m}$ and of $GL_{k}\left(\mathbb{F}\right)$.
Our purpose is twofold:  providing a necessary background from the
general theory, and presenting the connection between $W_{n,m}$
and the $GL_{k}\left(\mathbb{F}\right)$-representation $\mathbb{F}\left\langle x_{1},...,x_{k}\right\rangle \cap\left(x^{n}\right)^{T}$, established by Drensky and Benanti in \cite{benanti1999polynomial}.
This connection serves as a central tool in proving the main result of this paper (Theorem \ref{stab}). 
For more details regarding the general theory of $S_{m}$-representations, we refer the reader to \cite{henke235explicit}  or \cite{rowen2008graduate}, and to \cite{procesi2007lie} for the general theory of $GL_{k}\left(\mathbb{F}\right)$-representations.}

\subsection{The Representation Theory of \texorpdfstring{$S_{m}$}{Lg} }
\begin{definition}
\label{Partition}
A \textit{partition $\lambda$ }of $m$, denoted by $\lambda\vdash m$,
is a sequence of integers $\lambda=\left(\lambda_{1},\lambda_{2},...,\lambda_{r}\right)$
such that $\lambda_{1}\geq...\geq\lambda_{r}$ and $\lambda_{1}+\cdot\cdot\cdot+\lambda_{r}=m$. We call $r=\ell\left(\lambda\right)$ the \textbf{length} of the partition
$\lambda$. We also write
$\lambda=\left(1^{a_{1}},2^{a_{2}},...\right)$ to indicate the number
of times each integer occurs in the partition. For example, $\left(5,2,2,1,1,1\right)$
and $\left(1^{3},2^{2},5^{1}\right)$ represent the same partition.
\end{definition}

\begin{definition}
The \textbf{Young diagram} associated to $\lambda=\left(\lambda_{1},\lambda_{2},...,\lambda_{r}\right)\vdash m$
is the set \linebreak  $D_{\lambda}=\left\{ \left(i,j\right)\in\mathbb{Z}^{2}\mid i=1,...,r,j=1,...,\lambda_{i}\right\}$.
\end{definition} 

\nohyphens{It is convenient to display Young diagrams graphically, as the following
figure illustrates:}

\begin{figure}[h]
    \centering
    \includegraphics[width=0.25\textwidth]{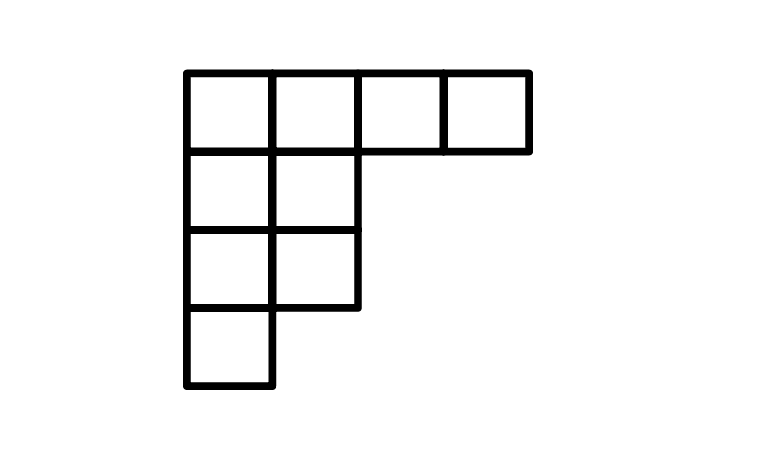}
    \caption{The Young diagram of $\lambda=\left(4,2,2,1\right)=\left(1,2^{2},4\right)$.}
    \label{fig:mesh1}
\end{figure}

\begin{definition}
\label{hookk}
Let $\lambda\vdash m$. The \textbf{hook} of the box $\left(i,j\right)\in D_{\lambda}$, is the set of boxes \linebreak $H_{i,j}^{\lambda}=\left\{ \left(i,j'\right)\in D_{\lambda}:j'\geq j\right\} \cup\left\{ \left(i',j\right)\in D_{\lambda}:i'\geq i\right\} $.
The \textbf{hook length} of $\left(i,j\right)$ is $h_{i,j}^{\lambda}=\left|H_{i,j}^{\lambda}\right|$.
\end{definition}
\begin{example}
\label{ex_will}
Taking $\lambda=\left(4,3,1\right)$, and writing the hook number $h_{i,j}^{\lambda}$ inside the (i, j)-th box, we obtain the following configuration: \begin{center}
$  {\scriptsize \ \begin{ytableau} 6 & 4 & 3  & 1\\  4& 2  & 1\\1    \end{ytableau}}  $
.
\par\end{center}
\end{example}
It is well knowm that all the irreducible representations of $S_{m}$
are in one-to-one correspondence with partitions of $m$. Throughout, we
denote these representations by $\left\{ S^{\lambda}\right\} _{\lambda\vdash m}$.
The following is a well-known formula for $\text{dim}\left(S^{\lambda}\right)$.

\begin{theorem}[\textbf{The hook length formula.} (See, e.g., {\cite[p.226]{rowen2008graduate}}) ]
\label{hooh}
Let $\lambda\vdash m$. The dimension of the irreducible
representation $S^{\lambda}$ is given by the formula:
\begin{gather*}
\text{dim}\left(S^{\lambda}\right)=\frac{m!}{\underset{\left(i,j\right)\in D_{\lambda}}{\prod}h_{i,j}^{\lambda}}.
\end{gather*}
\end{theorem}

\begin{definition}
\label{elemen}
A \textbf{Young tableau} $T_{\lambda}$  of \textbf{shape} $\lambda\vdash m$ is a Young diagram $D_{\lambda}$ filled
with the integers $1,...,m$ without repetitions. A Young tableau $T_{\lambda}$
is called \textbf{standard}, if the numbers in $T_{\lambda}$ increase
along the rows and down the columns. We denote by $\text{Tab}\left(\lambda\right)$
the set of all standard Young tableaux.

Let $T$ be a Young tableau of shape $\lambda$.
The \textbf{row subgroup} of $T$ is the subgroup $R_{T}\leq S_{m}$,
where $\sigma\in R_{T}$ if and only if $i$ and $\sigma\left(i\right)$
lie in the same row of $T$ for any $i\in\left[m\right]$. Similarly,
one defines the \textbf{column subgroup} $C_{T}$. We define the following elements in $\mathbb{F}\left[S_{m}\right]$:
\begin{flalign*}
a_{T}:= & \underset{\sigma\in R_{T}}{\sum}\sigma\\
b_{T}:= & \underset{\tau\in C_{T}}{\sum}\text{sgn}\left(\tau\right)\tau\\
e_{T}:= & a_{T}b_{T}=\underset{\sigma\in R_{T},\tau\in C_{T}}{\sum}\text{sgn}\left(\tau\right)\sigma\tau.
\end{flalign*}
\end{definition}

The following theorem summarizes some of the fundamental properties of $S_{m}$-representations.

\begin{theorem}[see {\cite{procesi2007lie}}]
\label{basic_sn}
Let $R=\mathbb{F}\left[S_{m}\right]$. 
For every $\lambda\vdash m$ and $T\in\text{Tab}\left(\lambda\right)$,
the following hold:

(a) $Re_{T}\cong S^{\lambda}$, where $S^{\lambda}$ is the irreducible
$S_{m}$-representation corresponding to $\lambda$.

(b) There is $\beta_{\lambda}\in\mathbb{N}$ such that $e_{T}^{2}=\beta_{\lambda}e_{T}$.

(c) For every $\sigma\in S_{m}$, $\sigma e_{T}\sigma^{-1}=e_{\sigma T}$,
where $\sigma T$ is the tableau obtained by applying $\sigma$ on
each entry of $T$. 

(d) $e_{T}Re_{T}=\mathbb{F}e_{T}$. In particular, $\text{dim}\left(e_{T}Re_{T}\right)=1$.
If $U,V$ are tableaux of different shapes $\lambda,\mu$, then $e_{U}Re_{V}=0$.

\textbf{(2) }$R=\underset{\lambda\vdash m}{\bigoplus}\left(\underset{T_{\lambda}\in\text{Tab}\left(\lambda\right)}{\bigoplus}Re_{T_{\lambda}}\right)$. 

\textbf{(3)} Letting $d^{\lambda}=\text{dim}\left(S^{\lambda}\right)$, we have the following isomorphism of $S_{m}$-representations: 
\begin{gather*}
 R\cong\underset{\lambda\vdash m}{\bigoplus}d^{\lambda}S^{\lambda}.  
\end{gather*}
\end{theorem}
\begin{remark}
\label{poll}
One can observe that $\text{dim}\left(e_{T}Re_{T'}\right)=1$ holds for any $T,T'\in\text{Tab}\left(\lambda\right)$. Indeed, taking
$\sigma\in S_{m}$ such that $T'=\sigma T$, we have by (c) and (d)
\begin{gather*}
e_{T}Re_{T'}=e_{T}Re_{\sigma T}=e_{T}R\sigma e_{T}\sigma^{-1}=\mathbb{F}e_{T}\sigma^{-1}.
\end{gather*}
\end{remark}
Recall from Section 1 that  $V_{m}\cap I$
is an $S_{m}$-representation for every $T$-ideal $I$. One way to calculate its decomposition
into $S_{m}$-irreducible representations is as follows.

\begin{lemma}
\label{mult_is_dim}
Let $I\triangleleft\mathbb{F}\left\langle X\right\rangle $
be a $T$-ideal, and let $\underset{\lambda\vdash m}{\bigoplus}m^{\lambda}S^{\lambda}$
be the $S_{m}$-decomposition of \linebreak $W=V_{m}\cap I$.
For every $\lambda\vdash m$, we have $m^{\lambda}=\text{dim}\left(e_{T_{\lambda}}W\right)$,
where $T_{\lambda}$ is any Young tableau of shape $\lambda$. In
particular, $m^{\lambda}=l$ if and only if there are $w_{1},...,w_{l}\in W$
and a Young tableau $T_{\lambda}$, such that $\left\{ e_{T_{\lambda}}w_{1},...,e_{T_{\lambda}}w_{l}\right\} $
is a basis of $e_{T_{\lambda}}W$.
\end{lemma} 

\begin{proof}
Let $R=\mathbb{F}\left[S_{m}\right]$ and fix a partition $\lambda$ and $T_{\lambda}\in\text{Tab}\left(\lambda\right)$.
Given $\mu\vdash m$ and an irreducible representation $U\subseteq\mathbb{F}\left[S_{m}\right]$
isomorphic to $S^{\mu}$, we have $U\cong R_{T_{\mu}}$ for (any)
$T_{\mu}\in\text{Tab}\left(\mu\right)$, and thus by Theorem \ref{basic_sn} and
 Remark \ref{poll},
\begin{gather*}
\text{dim}\left(e_{T_{\lambda}}S^{\mu}\right)=\text{dim}\left(e_{T_{\lambda}}Re_{T_{\mu}}\right)=\begin{cases}
0 & \mu\neq\lambda\\
1 & \mu=\lambda
\end{cases}.
\end{gather*}

It follows that
\begin{gather*}
\text{dim}\left(e_{T_{\lambda}}W\right)=\text{dim}\left(e_{T_{\lambda}}\left(\underset{\mu\vdash m}{\bigoplus}m^{\mu}Re_{T_{\mu}}\right)\right)=m^{\lambda}\text{dim}\left(e_{T_{\lambda}}Re_{T_{\lambda}}\right)=m^{\lambda}.
\end{gather*}
\end{proof}

\subsubsection{Substitutions and linearization}

In Section 1, we have identified $V_{m}$ with $\mathbb{F}\left[S_{m}\right]$ as an $\mathbb{F}\left[S_{m}\right]$-bimodule. Accordingly,
we may write any $a\in\mathbb{F}\left[S_{m}\right]$ as a polynomial $a\left(x_{1},...,x_{m}\right)\in V_{m}$.
For example, if 
\begin{center}
$  T={\scriptsize \ \begin{ytableau} 1 & 2 \\ 3   \end{ytableau}} , $
\par\end{center}

then $R_{T}=\left\{ e,\left(12\right)\right\} $, $C_{T}=\left\{ e,\left(13\right)\right\} $
and 
\begin{flalign*}
e_{T} & =e_{T}\left(x_{1},x_{2},x_{3}\right)=\left(e+\left(12\right)\right)\left(e-\left(13\right)\right)x_{1}x_{2}x_{3}\\
 & =x_{1}x_{2}x_{3}-x_{3}x_{2}x_{1}+x_{2}x_{1}x_{3}-x_{3}x_{1}x_{2}.
\end{flalign*}

We also recall that under this identification, 
when a permutation $\sigma\in S_{m}$ acts on a monomial in $V_{m}$ from the left, it substitutes the variables according to $\sigma$,
whereas when $\sigma\in S_{m}$ acts from the right, it permutes the \textit{places} of the variables according to $\sigma$ (see Equation (\ref{left_&_right})).

In what follows, we shall present the concepts of substitution and
linearization in the spirit of Regev's paper \cite{regev1980polynomial}. This approach will
simplify some calculations in Section 4. 

\begin{definition}
\label{Regevsub}
Fix a new countable set of variables $Y=\left\{ y_{1},y_{2},...\right\}$. The \textbf{substitution homomorphism} is  defined by 
\begin{flalign*}
\text{Sub}_{\left(\text{x},\text{y}\right)} & :\mathbb{F}\left\langle X\right\rangle \rightarrow\mathbb{F}\left\langle Y\right\rangle \\
x_{i} & \mapsto y_{i}.
\end{flalign*}
For ease of notations, we write polynomials in $\mathbb{F}\left\langle X\right\rangle $
by $f\left(\text{x}\right)$ instead of $f\left(x_{i_{1}},x_{i_{2},...}\right)$,
and similarly for $\mathbb{F}\left\langle Y\right\rangle $. In this
notation, we have $\text{Sub}_{\left(\text{x},\text{y}\right)}\left(f\left(\text{x}\right)\right)=f\left(\text{y}\right)$. 
\end{definition}

\begin{definition}
\label{Regev_sub_T}
 Let $\lambda\vdash m$, and $T\in\text{Tab}\left(\lambda\right)$.
Suppose that $a_{i1},a_{i2},a_{i3},...,a_{ih_{i}}$ are the elements
appearing in the $i$-th row of $T$. The \textbf{substitution homomorphism induced by $T$}
 is defined by
\begin{gather*}
\text{Sub}_{\left(\text{x},\text{y}\right)}^{T}:\mathbb{F}\left\langle x_{1},...,x_{m}\right\rangle \rightarrow\mathbb{F}\left\langle Y\right\rangle \\
x_{a_{i1}},x_{a_{i2}},x_{a_{i3}},...\mapsto y_{i}.
\end{gather*}
\end{definition} 
For example, the tableau $  T={\scriptsize \ \begin{ytableau} 1 & 3 & 5 \\ 2 & 4   \end{ytableau}}  $
induces the substitutions $x_{1},x_{3},x_{5}\mapsto y_{1}$ and  $x_{2},x_{4}\mapsto y_{2}$,
and so
\begin{gather*}
\text{Sub}_{\left(\text{x},\text{y}\right)}^{T}\left(x_{1}x_{2}x_{3}x_{4}x_{5}+x_{2}x_{1}x_{3}x_{4}x_{5}\right)=y_{1}y_{2}y_{1}y_{2}y_{1}+y_{2}y_{1}^{2}y_{2}y_{1}.
\end{gather*}

\begin{definition}
Let $\underline{h}=\left(h_{1},...,h_{k}\right)$ be a $k$-tuple
of nonnegative integers, such that \linebreak $h_{1}+\cdot\cdot\cdot+h_{k}=m$.  We define the subgroup 
\begin{gather*}
S_{\underline{h}}:=S_{1,...,h_{1}}\times\cdot\cdot\cdot\times S_{m-h_{k}+1,...,m}\subseteq S_{m},
\end{gather*}

where $S_{a,...,b}=\text{Sym}\left\{ a,...,b\right\}$.

Let $g=g\left(y_{1},...,y_{k}\right)$ be a homogeneous polynomial of multidegree
$\underline{h}$, (that is,  each monomial in $g$ has degree $h_{i}$ in $y_{i}$). The \textbf{linearization} $\text{Lin}_{m}\left(g\right)\in V_{m}$
is defined as follows.

If $g$ is the monomial $g=y_{1}^{h_{1}}\cdot\cdot\cdot y_{k}^{h_{k}}$,
then 
\begin{gather*}
\text{Lin}_{m}\left(y_{1}^{h_{1}}\cdot\cdot\cdot y_{k}^{h_{k}}\right)=\left(\underset{\sigma\in S_{\underline{h}}}{\sum}\sigma\right)x_{1}\cdot\cdot\cdot x_{m}\in V_{m}.
\end{gather*}

In general, if $g$ is any polynomial of multidegree $\underline{h}$,
we write $g=y_{1}^{h_{1}}\cdot\cdot\cdot y_{k}^{h_{k}}\cdot a$ for
some $a\in\mathbb{F}\left[S_{m}\right]$ (one should notice that this is always possible), and define 
\begin{gather*}
\text{Lin}_{m}\left(y_{1}^{h_{1}}\cdot\cdot\cdot y_{k}^{h_{k}}\cdot a\right)=\text{Lin}_{m}\left(y_{1}^{h_{1}}\cdot\cdot\cdot y_{k}^{h_{k}}\right)a.
\end{gather*}
\end{definition}

For example, if $g=y_{1}y_{2}y_{2}y_{1}$, then $g=y_{1}^{2}y_{2}^{2}\cdot\left(24\right)$,
and 
\begin{flalign*}
\text{Lin}_{m}\left(g\right) & =\text{Lin}_{m}\left(y_{1}^{2}y_{2}^{2}\right)\cdot\left(24\right)\\
 & =\left(x_{1}x_{2}x_{3}x_{4}+x_{2}x_{1}x_{3}x_{4}+x_{1}x_{2}x_{4}x_{3}+x_{2}x_{1}x_{4}x_{3}\right)\cdot\left(24\right)\\
 & =x_{1}x_{4}x_{3}x_{2}+x_{2}x_{4}x_{3}x_{1}+x_{1}x_{3}x_{4}x_{2}+x_{2}x_{3}x_{4}x_{1}.
\end{flalign*}

\begin{remark}
Formally, the domain of $\text{Lin}_{m}$ is the subspace of the free
algebra $\mathbb{F}\left\langle y_{1},...,y_{m}\right\rangle $, consisting 
of all homogeneous polynomials of degree $m$. It is also clear
that $\text{Lin}_{m}$ is a linear map.

The reader should be aware that according to our notation, for a given tableau $T$, $b_{T}\left(\text{y}\right)$
stands for $\text{Sub}_{\left(\text{x},\text{y}\right)}\left(b_{T}\left(\text{x}\right)\right)$,
where  $b_{T}=\underset{\tau\in C_{T}}{\sum}\text{sgn}\left(\tau\right)\tau$.
\begin{proposition}
\label{inter}
Fix $\lambda\vdash m$ and $T_{\lambda}\in\text{Tab}\left(\lambda\right)$.  For every $f=f\left(x_{1},...,x_{m}\right)\in \mathbb{F}\left[S_{m}\right]$,
we have \begin{gather*}
b_{T_{\lambda}}\left(y\right)f=\underset{\tau\in C_{T_{\lambda}}}{\sum}\text{sgn}\left(\tau\right)f\left(y_{i_{\tau\left(1\right)}},...,y_{i_{\tau\left(m\right)}}\right),
\end{gather*} where $y_{i_{j}}=\text{Sub}_{\left(\text{x},\text{y}\right)}^{T_{\lambda}}\left(x_{j}\right)$, 
$j=1,...,m$. (As usual, the notation $f=f\left(x_{1},...,x_{m}\right)$
uses the identification $\mathbb{F}\left[S_{m}\right]=V_{m}$.)
\end{proposition}

\begin{proof}
First, note that
\begin{gather*}
b_{T_{\lambda}}\left(y\right)=\text{Sub}_{\left(\text{x},\text{y}\right)}^{T_{\lambda}}\left(\underset{\tau\in C_{T_{\lambda}}}{\sum}\text{sgn}\left(\tau\right)x_{\tau\left(1\right)}\cdot\cdot\cdot x_{\tau\left(m\right)}\right)=\underset{\tau\in C_{T_{\lambda}}}{\sum}\text{sgn}\left(\tau\right)y_{i_{\tau\left(1\right)}}\cdot\cdot\cdot y_{i_{\tau\left(m\right)}}.
\end{gather*}

Since any $\sigma\in S_{m}$ acts from the right by $y_{i_{1}}\cdot\cdot\cdot y_{i_{m}}\sigma=y_{i_{\sigma\left(1\right)}}\cdot\cdot\cdot y_{i_{\sigma\left(m\right)}}$, 
for every \linebreak $f=f\left(x_{1},...,x_{m}\right)=\underset{\sigma\in S_{m}}{\sum}\alpha_{\sigma}\sigma\in\mathbb{F}\left[S_{m}\right]$
we have
\begin{gather*}
y_{i_{1}}\cdot\cdot\cdot y_{i_{m}}f=\underset{\sigma\in S_{m}}{\sum}\alpha_{\sigma}y_{i_{\sigma\left(1\right)}}\cdot\cdot\cdot y_{i_{\sigma\left(m\right)}}=f\left(y_{i_{1}},...,y_{i_{m}}\right),
\end{gather*}

implying that
\begin{gather*}
b_{T_{\lambda}}\left(y\right)f=\underset{\tau\in C_{T_{\lambda}}}{\sum}\text{sgn}\left(\tau\right)y_{i_{\tau\left(1\right)}}\cdot\cdot\cdot y_{i_{\tau\left(m\right)}}f=\underset{\tau\in C_{T_{\lambda}}}{\sum}\text{sgn}\left(\tau\right)f\left(y_{i_{\tau\left(1\right)}},...,y_{i_{\tau\left(m\right)}}\right).
\end{gather*}
\end{proof}

\end{remark}
Let us record a remarkable feature of $\text{Lin}_{m}$,
which was first observed by Regev.

\begin{proposition}[L.2.4 and Theorem 2.8 in \cite{regev1982polynomial}]
\label{Regev_sub}
Given $\lambda=\left(\lambda_{1},\lambda_{2},...\right)\vdash m$ and $T_{\lambda}\in\text{Tab}\left(\lambda\right)$,
we let $T\in\text{Tab}\left(\lambda\right)$ denote the tableau 

\begin{figure}[H]
    \centering
    \includegraphics[width=0.2\textwidth]{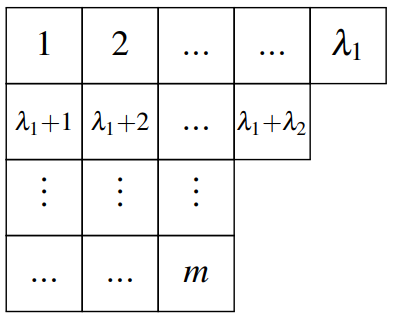},
\end{figure}

and consider the permutation $\sigma\in S_{m}$  satisfying $\sigma T_{\lambda}=T$.
For every $g\in\mathbb{F}\left[S_{m}\right]$, we have \linebreak
$\text{Lin}_{m}\left(b_{T_{\lambda}}\left(\text{y}\right)g\right)=\sigma e_{T_{\lambda}}\left(\text{x}\right)g$.

\end{proposition}

Recall from Lemma \ref{mult_is_dim} that the multiplicity of $S^{\lambda}$ in $W=V_{m}\cap I$
(for a $T$-ideal $I\triangleleft\mathbb{F}\left\langle X\right\rangle $)
is the dimension of $e_{T_{\lambda}}W$, for any $T_{\lambda}\in\text{Tab}\left(\lambda\right)$. This dimension can be calculated more efficiently by the following corollary.

\begin{corollary}
\label{regev_col}
Fix $\lambda\vdash m$ and $T_{\lambda}\in\text{Tab}\left(\lambda\right)$,
and let $f_{1},...,f_{l}\in\mathbb{F}\left[S_{m}\right]$. The following are equivalent:

\textbf{(i)} The elements $e_{T_{\lambda}}f_{1},...,e_{T_{\lambda}}f_{l}\in\mathbb{F}\left\langle X\right\rangle$
are linearly independent.

\textbf{(ii)} The elements $b_{T_{\lambda}}\left(y\right)f_{1},...,b_{T_{\lambda}}\left(y\right)f_{l}\in\mathbb{F}\left\langle Y\right\rangle $
are linearly independent.

\textbf{(iii)} The elements $\underset{\tau\in C_{T_{\lambda}}}{\sum}\text{sgn}\left(\tau\right)f_{1}\left(y_{i_{\tau\left(1\right)}},...,y_{i_{\tau\left(m\right)}}\right),...,\underset{\tau\in C_{T_{\lambda}}}{\sum}\text{sgn}\left(\tau\right)f_{l}\left(y_{i_{\tau\left(1\right)}},...,y_{i_{\tau\left(m\right)}}\right)$
are linearly independent (where $y_{i_{j}}=\text{Sub}_{\left(\text{x},\text{y}\right)}^{T_{\lambda}}\left(x_{j}\right)$ as above). 

\end{corollary}

\begin{proof} ((i) $\Rightarrow$ (ii))  Suppose $e_{T_{\lambda}}f_{1},...,e_{T_{\lambda}}f_{l}$
are linearly independent, and let $a_{1},...,a_{l}\in\mathbb{F}$
such that
\begin{align}
  a_{1}b_{T_{\lambda}}\left(\text{y}\right)f_{1}+\cdot\cdot\cdot+a_{l}b_{T_{\lambda}}\left(\text{y}\right)f_{l}=0. \label{independent}
\end{align}

Applying (the linear map) $\text{Lin}_{m}$ to both sides of (\ref{independent}),
we obtain by Proposition \ref{Regev_sub}
\begin{gather*}
\sigma\left(a_{1}e_{T_{\lambda}}\left(x\right)f_{1}+\cdot\cdot\cdot+a_{l}e_{T_{\lambda}}\left(x\right)f_{l}=0\right)
\end{gather*}

for some $\sigma\in S_{m}$, and thus $a_{1}=...=a_{l}=0$.

The converse ((i) $\Leftarrow$ (ii)) is immediate from $\text{Lin}_{m}$ being a linear
map. Finally, (ii) $\Leftrightarrow$ (iii) follows directly from Proposition \ref{inter}. 
\end{proof}

\subsection{The Representation Theory of  \texorpdfstring{$GL_{k}\left(\mathbb{F}\right)$}{Lg} }

Throughout this section, we fix $k$, and consider the free (non-unitary) algebra of rank
$k$, \linebreak $A_{k}:=\mathbb{F}\left\langle x_{1},...,x_{k}\right\rangle $.
This algebra can be graded as $A_{k}=\stackrel[m\geq1]{}{\bigoplus}\left[A_{k}\right]^{\left(m\right)}$,
where for every $m\geq1$, $\left[A_{k}\right]^{\left(m\right)}$ is the subspace
spanned by all the monomials of degree $m$. Letting 
\begin{align}
  V:=\text{span}_{\mathbb{F}}\left\{ x_{1},...,x_{k}\right\},
\label{v}
\end{align}
we may identify $\left[A_{k}\right]^{\left(m\right)}$ and $V^{\otimes m}$ as
vector spaces via $x_{i_{1}}\cdot\cdot\cdot x_{i_{m}}\leftrightarrow x_{i_{1}}\otimes\cdot\cdot\cdot\otimes x_{i_{m}}$.
Since  $GL_{k}\left(\mathbb{F}\right)\cong GL\left(V\right)$ acts
canonically on $V$ as linear transformations, $A_{k}$ becomes $a$ $GL_{k}\left(\mathbb{F}\right)$-representation by extending this action diagonally: 
\begin{gather*}
g.\left(x_{i_{1}}\cdot\cdot\cdot x_{i_{s}}\right)=g\left(x_{i_{1}}\right)\cdot\cdot\cdot g\left(x_{i_{s}}\right).
\end{gather*}

The $GL_{k}\left(\mathbb{F}\right)$-representations obtained this
way in $A_{k}$ are called \textbf{polynomial representations}. Given a polynomial
representation $W$, we write $\text{deg}\left(W\right)=m$ if $W\subset\left[A_{k}\right]^{\left(m\right)}$,
indicating that all monomials in $W$ have degree $m$.

There is also a natural action of $S_{m}$ on $V^{\otimes m}$ by
permuting the places of each monomial:
\begin{align}
  \sigma\left(x_{i_{1}}\cdot\cdot\cdot x_{i_{m}}\right)=x_{i_{\sigma^{-1}\left(1\right)}}\cdot\cdot\cdot x_{i_{\sigma^{-1}\left(m\right)}}. \label{prem_r}
\end{align}
Extending this action to all of $\mathbb{F}\left[S_{m}\right]$,
we can give the following definition.
\begin{definition}
Let $V$ be the $k$-dimensional vector space as in (\ref{v}). Given $\lambda\vdash m$
and a tableau $T_{\lambda}\in\text{Tab}\left(\lambda\right)$, we define
the \textbf{Weyl module} of degree $m$ associated to $\lambda$ to
be 
\begin{gather*}
V^{\lambda}:=e_{T_{\lambda}}V^{\otimes m}.
\end{gather*}
\end{definition}

A celebrated result of Schur and Weyl describes the irreducible polynomial
$GL_{k}\left(\mathbb{F}\right)$-representations as follows:

\begin{theorem}[see {\cite[p.252]{procesi2007lie} or \cite[p.366]{kanel2015computational} }]
\label{Schur} Let $V^{\lambda}$ be the Weyl $GL_{k}\left(\mathbb{F}\right)$-module as above. 

\textbf{(1)} $V^{\lambda}$ is an irreducible $GL_{k}\left(\mathbb{F}\right)$-representation,
and every irreducible polynomial representation of degree $m$ is
isomorphic to $V^{\mu}$, for some $\mu\vdash m$. 

\textbf{(2)} If $S,T$ are two tableaux of the same shape $\lambda$,
then $e_{S}V^{\otimes m}\cong e_{T}V^{\otimes m}$ as $GL_{k}\left(\mathbb{F}\right)$-representations.

\textbf{(3)} If $\ell\left(\lambda\right)>k$ (see definition \ref{Partition}),
then $V^{\lambda}=0$. Moreover, we have the following isomorphism
of $GL_{k}\left(\mathbb{F}\right)$-representations: 
\begin{gather*}
V^{\otimes m}\cong\underset{\underset{\ell\left(\lambda\right)\leq k}{\lambda\vdash m}}{\bigoplus}d^{\lambda}V^{\lambda},
\end{gather*}

where $d^{\lambda}$ is the dimension of (the irreducible $S_{m}$-representation)
$S^{\lambda}$.

\end{theorem}

\begin{remark}
One should observe that for any $T$-ideal $I$, $\left[A_{k}\right]^{\left(m\right)}\cap I$ is also a $GL_{k}\left(\mathbb{F}\right)$-representation.
This is because $I$ is invariant under substitutions. Thus, just as with the decomposition of $V_{m}\cap I$ into irreducible $S_{m}$-representations
(established in Section 1), we have a decomposition of $\left[A_{k}\right]^{\left(m\right)}\cap I$ into $GL_{k}\left(\mathbb{F}\right)$-representations:
\begin{gather*}
A_{k}^{\left(m\right)}\cap I\cong\underset{\underset{\ell\left(\lambda\right)\leq k}{\lambda\vdash m}}{\bigoplus}n^{\lambda}V^{\lambda}.
\end{gather*}
\end{remark}
In the early 1980s, Berele and Drensky established the following connection between the  $GL_{k}\left(\mathbb{F}\right)$-decomposition
of $\left[A_{k}\right]^{\left(m\right)}\cap I$ and the $S_{m}$-decomposition
of $V_{m}\cap I$:

\begin{theorem}[Berele \cite{berele1982homogeneous}, Drensky \cite{drensky1984codimensions}]
\label{Connection_Sn_Glk}
Let $I\triangleleft\mathbb{F}\left\langle X\right\rangle $ be a $T$-ideal,
and suppose that as an $S_{m}$-representation, $V_{m}\cap I$ decomposes as $\underset{\lambda\vdash m}{\bigoplus}q^{\lambda}S^{\lambda}$.
Then, as a $GL_{k}\left(\mathbb{F}\right)$-representation, $\left[A_{k}\right]^{\left(m\right)}\cap I$
decomposes as
\begin{gather*}
\underset{\underset{\ell\left(\lambda\right)\leq k}{\lambda\vdash m}}{\bigoplus}q^{\lambda}V^{\lambda}.
\end{gather*}
\end{theorem}
In other words, the space $\left[A_{k}\right]^{\left(m\right)}\cap I$ "captures"
the multiplicities $\left\{ q^{\lambda}:\ell\left(\lambda\right)\leq k\right\} $
of $V_{m}\cap I$.

\subsection{Symmetric tensors} 

Given an $\mathbb{F}$-vector space, the \textit{$n$-th} \textit{symmetric
power} of $U$, denoted $U^{\otimes_{s}n}=\overset{n\text{ times}}{\overbrace{U\otimes_{s}\cdot\cdot\cdot\otimes_{s}U}}$,
is the vector space 
\begin{gather*}
U^{\otimes_{s}n}=U^{\otimes n}/\left\langle \sigma\cdot w-w:w\in U^{\otimes n},\sigma\in S_{n}\right\rangle _{\mathbb{F}}.
\end{gather*}

Elements in $U^{\otimes_{s}n}$ can be written as $u_{1}\otimes_{s}\cdot\cdot\cdot\otimes_{s}u_{n}$,
$u_{i}\in U$, where we identify $u_{1}\otimes_{s}\cdot\cdot\cdot\otimes_{s}u_{n}$
and $u_{\sigma\left(1\right)}\otimes_{s}\cdot\cdot\cdot\otimes_{s}u_{\sigma\left(n\right)}$
for all $\sigma\in S_{n}$. We denote the space of all the polynomials in  $U^{\otimes_{s}n}$
of degree $m$ (i.e., the polynomials whose pre-image in $U^{\otimes n}$ under the quotient map has degree $m$) by $\left[U^{\otimes_{s}n}\right]^{\left(m\right)}$. 

\begin{remark}
\label{remark_sym} 
Let $V=\text{span}_{\mathbb{F}}\left\{ x_{1},...,x_{k}\right\} $.
It is easy to see that $V^{\otimes_{s}n}$ is isomorphic (as a $GL_{k}$-representation)
to $V^{\left(n\right)}=e_{T}V^{\otimes n}$, where $T$ is the
tableau $  T=$ ${\tiny \   \begin{ytableau} 1 & 2 & ... & n  \end{ytableau}}$ ,
via the map 
\begin{gather*}
x_{i_{1}}\otimes_{s}\cdot\cdot\cdot\otimes_{s}x_{i_{n}}\mapsto\underset{\sigma\in S_{n}}{\sum}x_{i_{\sigma\left(1\right)}}\otimes\cdot\cdot\cdot\otimes x_{i_{\sigma\left(n\right)}}.
\end{gather*}
\end{remark}

The following lemma of Drensky and Benanti, gives a formula for the decomposition of $U^{\otimes_{s}n}$ into irreducible representations.

\begin{lemma}[{\cite[Proposition 1.2]{benanti1999polynomial}}]
\label{decomp_of_sym}
Suppose $U=\stackrel[r=1]{\infty}{\bigoplus}V_{r}$ is a direct sum
of irreducible polynomial $GL_{k}\left(\mathbb{F}\right)$-representations,
such that $\text{deg}\left(V_{1}\right)\leq\text{deg}\left(V_{2}\right)\leq...$,
and for each partition $\lambda$, only a finite number of $V_{r}$
are isomorphic to $V^{\lambda}$.

\textbf{(1)} We have the following isomorphism of $GL_{k}\left(\mathbb{F}\right)$-representations:

\begin{align}
  U^{\otimes_{s}n}\cong\bigoplus_{r_{1},...,r_{q}}\bigoplus_{a_{1},...,a_{q}}\left(V_{r_{1}}^{\otimes_{s}a_{1}}\otimes V_{r_{2}}^{\otimes_{s}a_{2}}\otimes\cdot\cdot\cdot\otimes V_{r_{q}}^{\otimes_{s}a_{q}}\right),
\label{deomp_dren}
\end{align}
where the summation runs over all $r_{1}<...<r_{q}$ and $a_{1},...,a_{q}\in\mathbb{N}$
($q\leq n$) such that \linebreak $a_{1}+...+a_{q}=n$.

\textbf{(2)} Hence, given $m\in\mathbb{N}$, $\left[U^{\otimes_{s}n}\right]^{\left(m\right)}$
is isomorphic to a finite direct sum $\underset{i}{\bigoplus}Q_{i}$, where
we take from (\ref{deomp_dren}) only the summands $Q_{i}=V_{r_{1}}^{\otimes_{s}a_{1}}\otimes V_{r_{2}}^{\otimes_{s}a_{2}}\otimes\cdot\cdot\cdot\otimes V_{r_{q}}^{\otimes_{s}a_{q}}$
with total degree $m$, i.e., those that satisfy $a_{1}\text{deg}\left(V_{r_{1}}\right)+...+a_{q}\text{deg}\left(V_{r_{q}}\right)=m$. 
\end{lemma}

\begin{example}
The decomposition of $A_{2}=\mathbb{F}\left\langle x_{1},x_{2}\right\rangle $
is given by
\begin{gather*}
A_{2}=\bigoplus_{m=1}^{\infty}\left[A_{2}\right]^{\left(m\right)}\cong\bigoplus_{m=1}^{\infty}\underset{\underset{\ell\left(\lambda\right)\leq2}{\lambda\vdash m}}{\bigoplus}d^{\lambda}V^{\lambda}=V^{\left(1\right)}\oplus V^{\left(2\right)}\oplus V^{\left(1,1\right)}\oplus V^{\left(3\right)}\oplus2V^{\left(2,1\right)}\oplus V^{\left(4\right)}\cdot\cdot\cdot
\end{gather*}

which implies
\begin{flalign*}
A_{2}^{\otimes_{s}2} & \cong\left(V^{\left(1\right)}\right)^{\otimes_{s}2}\oplus\left(V^{\left(2\right)}\right)^{\otimes_{s}2}\oplus\cdot\cdot\cdot\oplus\left(\left(V^{\left(1\right)}\right)^{\otimes_{s}1}\otimes\left(V^{\left(2\right)}\right)^{\otimes_{s}1}\right)\oplus\cdot\cdot\cdot\\
\left[A_{2}^{\otimes_{s}2}\right]^{\left(3\right)} & \cong\left(\left(V^{\left(1\right)}\right)^{\otimes_{s}1}\otimes\left(V^{\left(2\right)}\right)^{\otimes_{s}1}\right)\oplus\left(\left(V^{\left(1\right)}\right)^{\otimes_{s}1}\otimes\left(V^{\left(1,1\right)}\right)^{\otimes_{s}1}\right).
\end{flalign*}
\end{example}

\begin{lemma}
\label{M_lemma}
Fix $K\in\mathbb{N}$. There are $GL_{k}\left(\mathbb{F}\right)$-representations
$M_{1},...,M_{p}\subseteq A_{k}$ of degree at most $2K$, and non-negative
integers $c_{1},...,c_{p}$, such that

\textbf{(1)} $\left[A_{k}^{\otimes_{s}K}\right]^{\left(2K\right)}\cong\stackrel[i=1]{p}{\bigoplus}V^{\left(c_{i}\right)}\otimes M_{i}$.

\textbf{(2)} For every $K<n$, $\left[A_{k}^{\otimes_{s}n}\right]^{\left(n+K\right)}\cong\stackrel[i=1]{p}{\bigoplus}V^{\left(n-K+c_{i}\right)}\otimes M_{i}$.

\end{lemma}

\begin{proof}
By Lemma \ref{decomp_of_sym}, $\left[A_{k}^{\otimes_{s}n}\right]^{\left(n+K\right)}$
decomposes as a direct sum of subrepresentations of the form 
\begin{gather*}
V_{1}^{\otimes_{s}a_{1}}\otimes\cdot\cdot\cdot\otimes V_{q}^{\otimes_{s}a_{q}},
\end{gather*}

where $V_{1},...,V_{q}$ are irreducible representations arising from
the decomposition of $A_{k}$, such that
\begin{align}
\begin{split}\label{eq:1}
    a_{1}+...+a_{q} & =n
\end{split}\\
\begin{split}\label{eq:2}
    a_{1}\text{deg}\left(V_{1}\right)+...+a_{q}\text{deg}\left(V_{q}\right) & =n+K\,\,\,\,\,\,\,\,\,\,\left(\text{deg}\left(V_{1}\right)\leq....\leq\text{deg}\left(V_{q}\right)\right).
\end{split}    
\end{align}
\textbf{Claim}. For every $n>K$, if $V_{1}^{\otimes_{s}a_{1}}\otimes V_{2}^{\otimes_{s}a_{2}}\otimes\cdot\cdot\cdot\otimes V_{q}^{\otimes_{s}a_{q}}$
is a summand in $\left[A_{k}^{\otimes_{s}n}\right]^{\left(n+K\right)}$,
then  $a_{1}$ can be written as \begin{gather*}
a_{1}=n-K+c_{1}
\end{gather*}
for some nonzero integer $c_{1}$, such that  $V_{1}^{\otimes_{s}c_{1}}\otimes V_{2}^{\otimes_{s}a_{2}}\cdot\cdot\cdot\otimes V_{q}^{\otimes_{s}a_{q}}$
is a summand in $\left[A_{k}^{\otimes_{s}K}\right]^{\left(2K\right)}$.

\textit{Proof}. Suppose $V_{1}^{\otimes_{s}a_{1}}\otimes V_{2}^{\otimes_{s}a_{2}}\otimes\cdot\cdot\cdot\otimes V_{q}^{\otimes_{s}a_{q}}$
is a summand in $\left[A_{k}^{\otimes_{s}n}\right]^{\left(n+K\right)}$.
Due to Theorem $\ref{Schur}$, we have 
\begin{gather*}
A_{k}=\left[A_{k}\right]^{\left(1\right)}\oplus\left[A_{k}\right]^{\left(2\right)}\oplus\cdot\cdot\cdot=\left(V^{\left(1\right)}\right)\oplus\left(V^{\left(2\right)}\oplus V^{\left(1,1\right)}\right)\oplus\cdot\cdot\cdot,
\end{gather*}

so $V^{\left(1\right)}=V=\text{span}_{\mathbb{F}}\left\{ x_{1},...,x_{k}\right\}$ is the only candidate for an irreducible representation
of degree one, and $\text{deg}\left(V_{i}\right)\geq2$ for all
$i\geq2$. Thus 
\begin{flalign*}
n+K & =a_{1}\text{deg}\left(V_{1}\right)+a_{2}\text{deg}\left(V_{2}\right)+...+a_{q}\text{deg}\left(V_{q}\right)\\
 & \geq a_{1}\text{deg}\left(V_{1}\right)+2\left(a_{2}+...+a_{q}\right)=a_{1}\text{deg}\left(V_{1}\right)+2\left(n-a_{1}\right)
\end{flalign*}

and $a_{1}\left(2-\text{deg}\left(V_{1}\right)\right)\geq n-K>0$.
This implies $\text{deg}\left(V_{1}\right)=1$ (i.e., $V_{1}=V$), and $a_{1}\geq n-K$. Writing $a_{1}=n-K+c_{1}$ for some $c_{1}\geq0$, we have \begin{gather*}
c_{1}+a_{2}+...+a_{q}=\left(K-n+a_{1}\right)+a_{2}+...+a_{q}\stackrel{\text{by (\ref{eq:1})}}{=}K
\end{gather*}
and 
\begin{flalign*}
\text{deg}\left(V_{1}^{\otimes_{s}c_{1}}\otimes V_{2}^{\otimes_{s}a_{2}}\cdot\cdot\cdot\otimes V_{q}^{\otimes_{s}a_{q}}\right) & =\text{deg}\left(V_{1}^{\otimes_{s}a_{1}}\otimes V_{2}^{\otimes_{s}a_{2}}\cdot\cdot\cdot\otimes V_{q}^{\otimes_{s}a_{q}}\right)-\left(n-K\right)\\
 & \stackrel{\text{by (\ref{eq:2})}}{=}2K,
\end{flalign*}

and thus $V_{1}^{\otimes_{s}c_{1}}\otimes V_{2}^{\otimes_{s}a_{2}}\cdot\cdot\cdot\otimes V_{q}^{\otimes_{s}a_{q}}$
is a summand in $\left[A_{k}^{\otimes_{s}K}\right]^{\left(2K\right)}$.

In light of this claim, it is not difficult to see that for any $n>K$, there is a bijection
between the summands in $\left[A_{k}^{\otimes_{s}n}\right]^{\left(n+K\right)}$
$ $ and the summands in $\left[A_{k}^{\otimes_{s}K}\right]^{\left(2K\right)}$,
given by 

\begin{gather*}
V_{1}^{\otimes_{s}a_{1}}\otimes V_{2}^{\otimes_{s}a_{2}}\otimes\cdot\cdot\cdot\otimes V_{q}^{\otimes_{s}a_{q}}\mapsto V_{1}^{\otimes_{s}a_{1}-\left(n-K\right)}\otimes V_{2}^{\otimes_{s}a_{2}}\otimes\cdot\cdot\cdot\otimes V_{q}^{\otimes_{s}a_{q}},
\end{gather*}

with an inverse 
\begin{gather*}
V_{1}^{\otimes_{s}b_{1}}\otimes V_{2}^{\otimes_{s}b_{2}}\otimes\cdot\cdot\cdot\otimes V_{l}^{\otimes_{s}b_{l}}\mapsto V_{1}^{\otimes_{s}b_{1}+\left(n-K\right)}\otimes V_{2}^{\otimes_{s}b_{2}}\otimes\cdot\cdot\cdot\otimes V_{l}^{\otimes_{s}b_{l}}.
\end{gather*}

All in all, assuming that the decomposition of $\left[A_{k}^{\otimes_{s}K}\right]^{\left(2K\right)}$
is given by 
\begin{gather*}
\stackrel[i=1]{p}{\bigoplus}V_{1}^{\otimes_{s}c_{i}}\otimes M_{i}\stackrel{\text{(\ref{remark_sym})}}{\cong}\stackrel[i=1]{p}{\bigoplus}V^{\left(c_{i}\right)}\otimes M_{i},
\end{gather*}

where each $M_{i}$ is of the form $V_{2}^{\otimes_{s}a_{2}}\otimes\cdot\cdot\cdot\otimes V_{q}^{\otimes_{s}a_{q}}$
($\text{deg}\left(V_{i}\right)>1$),  the decomposition of $\left[A_{k}^{\otimes_{s}n}\right]^{\left(n+K\right)}$ ($n>K$)
is given by
\begin{gather*}
\stackrel[i=1]{p}{\bigoplus}V_{1}^{\otimes_{s}n-K+c_{i}}\otimes M_{i}\stackrel{\text{(\ref{remark_sym})}}{\cong}\stackrel[i=1]{p}{\bigoplus}V^{\left(n-K+c_{i}\right)}\otimes M_{i}.
\end{gather*}

As $\text{deg}\left(M_{i}\right)\leq2K$ for each $i$, the lemma
follows.
\end{proof}

\subsection{Applying Young's rule}
 
Given two partitions $\mu\vdash l$ and $\lambda\vdash m$,
one may wish to decompose $V^{\lambda}\otimes V^{\mu}$ as a $GL_{k}\left(\mathbb{F}\right)$-representation,
that is, to compute the multiplicities $\left\{ c_{\lambda,\mu}^{\nu}\right\} $
in 
\begin{gather*}
V^{\mu}\otimes V^{\lambda}\cong\bigoplus_{\nu\vdash m+l}c_{\mu,\lambda}^{\nu}V^{\nu}.
\end{gather*}

\nohyphens{The \textit{Littlewood-Richardson rule} is a well-known algorithm
for doing this. For our purposes, we will only consider one special
case of this rule, usually referred to as the Young rule:}

\begin{theorem}[Young's rule. See, e.g., {\cite[p.88]{james2006representation}}]

Let $l,n\in\mathbb{N}$ and $\lambda\vdash n$. For every $\nu\vdash l+n$,
the multiplicity of $V^{\nu}$ in $V^{\left(l\right)}\otimes V^{\lambda}$
equals 
\begin{gather*}
c_{\left(l\right),\lambda}^{\nu}=\begin{cases}
1 & \begin{array}{c}
\text{if the diagram of \ensuremath{\nu} is obtained by adding \ensuremath{l} boxes to the diagram }\\
\text{of \ensuremath{\lambda}, such that no two new boxes are in the same column;}
\end{array}.\\
0 & \text{otherwise}
\end{cases}
\end{gather*}
\end{theorem}

Identifying $V^{\mu}$ with the Young diagram $D_{\mu}$, Figure \ref{fig:mesh2} illustrates Young's rule for the decomposition of $V^{\left(2\right)}\otimes V^{\left(2,1\right)}$. 

\begin{figure}[h]
    \centering
    \includegraphics[width=0.7\textwidth]{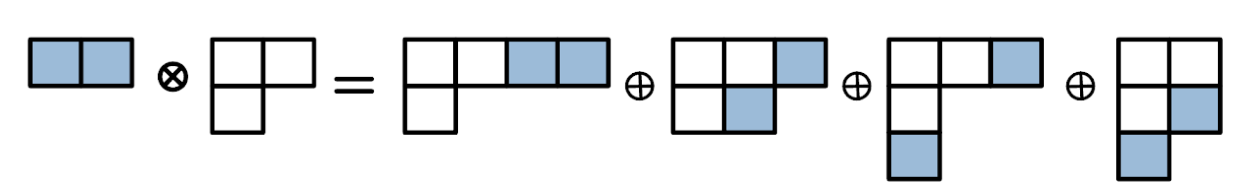}
    \caption{Young's rule for $V^{\left(2\right)}\otimes V^{\left(2,1\right)}$.}
    \label{fig:mesh2}
\end{figure}

\begin{definition}
\label{derived_GL}
In analogy with Definition \ref{derived_def}, we say that a $GL_{k}\left(\mathbb{F}\right)$-representation
$U$ \textit{is derived} from a $GL_{k}\left(\mathbb{F}\right)$-representation
$W$, if $W$ decomposes as $\underset{\mu}{\bigoplus}a_{\mu}V^{\mu}$,
and there exists $d\in\mathbb{N}$ such that $U$ decomposes as $\underset{\mu}{\bigoplus}a_{\mu}V^{\mu^{\left(d\right)}}$. (The notation $\mu^{\left(d\right)}$ was defined in Definition  $\ref{derived}$.)
\end{definition}

\begin{remark}
\label{simplest}
One can easily observe the following:

\textbf{(1)} If $U_{1}$ and $U_{2}$ are derived from $W_{1}$ and $W_{2}$,
respectively, then  $U_{1}\oplus U_{2}$ is derived from $W_{1}\oplus W_{2}$.

\textbf{(2)} Transitivity of derivations:  if $U$ is derived from $W$
and $W$ is derived from $Q$, then $U$ is derived from $Q$.
\end{remark}

Using the terminology of Definition \ref{derived_GL}, the following lemma is a simple application of Young's rule.
\begin{lemma}
\label{mul_derived}
Let $\lambda=\left(\lambda_{1},...,\lambda_{r}\right)$ be a partition
of $m$, and let $l$ be a natural number such that $\lambda_{1}\leq l$
(in particular, this is the case if $\left|\lambda\right|=m\leq l$). Then  $V^{\left(l+1\right)}\otimes V^{\lambda}$ is derived from $V^{\left(l\right)}\otimes V^{\lambda}$.
\end{lemma}

\begin{proof}
By Young's rule, $c_{\left(l+1\right),\lambda}^{\mu}=1$ if and only
if $D_{\mu}$ is obtained by adding $l+1$ boxes to the Young diagram $D_{\lambda}$ such that no two new boxes are in the same column.
As $\lambda_{1}\leq l$, $D_{\lambda}$ has at most $l$ columns,
so in every addition of at least $l+1$ new boxes in a legitimate manner, at least one will be added at the first row.
\end{proof}
This lemma is illustrated in Figure \ref{fig:mesh3}.
\begin{figure}[H]
    \centering
    \includegraphics[width=0.7\textwidth]{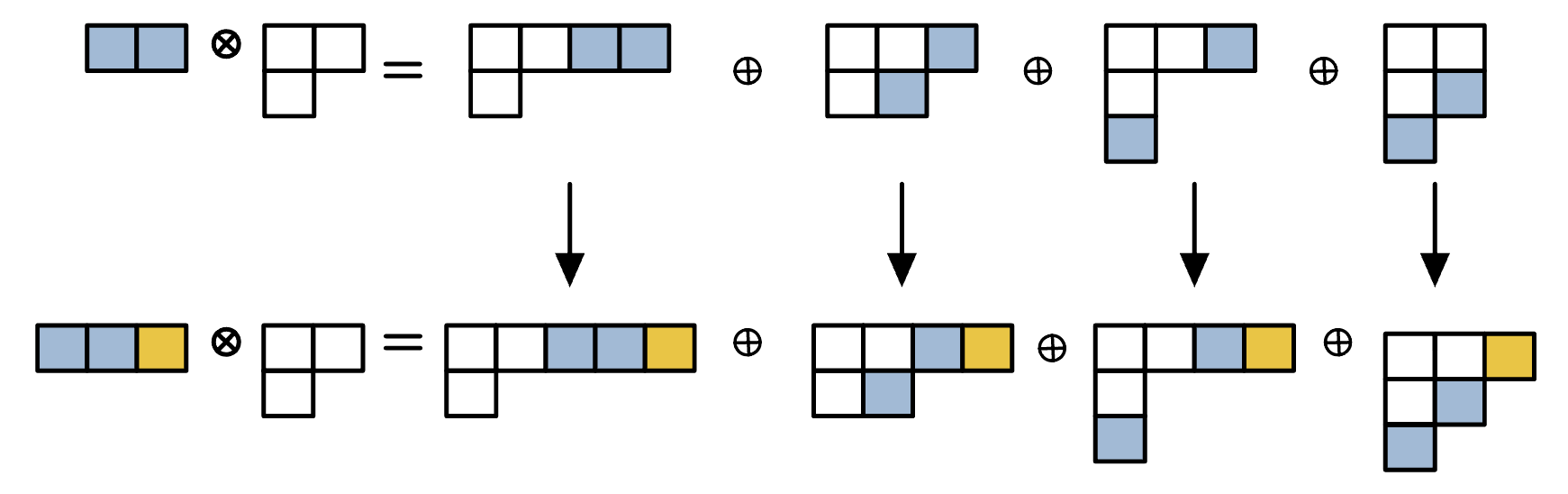}
    \caption{ $V^{\left(3\right)}\otimes V^{\left(2,1\right)}$
  is derived from $V^{\left(2\right)}\otimes V^{\left(2,1\right)}$}
    \label{fig:mesh3}
\end{figure}

\begin{corollary}
\label{der_M}
Suppose a $GL_{k}\left(\mathbb{F}\right)$-representation $M$
decomposes into irreducible representations as
$M=\stackrel[i=1]{q}{\bigoplus}V^{\lambda\left(i\right)}$,
and let $l$ be a natural number such that, for each  $i=1,...,q$, the degree of $V^{\lambda\left(i\right)}$
is at most $l$.  Then

\textbf{(1)} For every $a\in\mathbb{N}$,
if $\mu$ is a partition with $\ell\left(\lambda\right)>l+1$,
the multiplicity of $V^{\lambda}$
in $V^{\left(a\right)}\otimes M$ is zero.

\textbf{(2)}  $V^{\left(l+1\right)}\otimes M$ is derived
from $V^{\left(l\right)}\otimes M$.

\end{corollary}

\begin{proof}
By assumption, for every $i=1,...,q$~ the diagram of $\lambda\left(i\right)$
has at most $l$ rows. Hence, for every $a\in\mathbb{N}$, if $V^{\mu}$
is a summand in $V^{\left(a\right)}\otimes V^{\lambda\left(i\right)}$
then the diagram of $\mu$ has at most $l+1$ rows by Young's rule,
and \textbf{(1)} is proved.

For \textbf{(2)}, note that according to Lemma $\ref{mul_derived}$, 
 $V^{\left(l+1\right)}\otimes V^{\lambda\left(i\right)}$ is derived
from $V^{\left(l\right)}\otimes V^{\lambda\left(i\right)}$
for each $i$. Hence, by part \textbf{(1)} of Remark \ref{simplest}, the decomposition of $V^{\left(l+1\right)}\otimes M$
is derived from $V^{\left(l\right)}\otimes M$
as well. 
\end{proof}

Now we can state the main result of this section:

\begin{theorem}
\label{pre_main}
Fix $K\in\mathbb{N}$.

\textbf{(1)} For every $n\geq K$, if $\mu$ is a partition with $\ell\left(\mu\right)>2K+1$,
the multiplicity of $V^{\mu}$ in $\left[A_{k}^{\otimes_{s}n}\right]^{\left(n+K\right)}$
is zero.

\textbf{(2)} For every $n\geq3K$, $\left[A_{k}^{\otimes_{s}n}\right]^{\left(n+K\right)}$
is derived from $\left[A_{k}^{\otimes_{s}3K}\right]^{\left(4K\right)}$.

\textbf{(3)}  Let $\underset{\lambda\vdash n+K}{\bigoplus}l_{n,n+K}^{\lambda}V^{\lambda}$
be the decomposition of $\left[A_{k}^{\otimes_{s}n}\right]^{\left(n+K\right)}$.
There exists an integer $D>0$, such that for every $n\in\mathbb{N}$ 
\begin{gather*}
\underset{\lambda\vdash n+K}{\sum}l_{n,n+K}^{\lambda}\leq D.
\end{gather*}

\end{theorem}

\begin{proof}

\textbf{(1)}  By Lemma \ref{M_lemma}, $\left[A_{k}^{\otimes_{s}n}\right]^{\left(n+K\right)}$
decomposes as a direct sum of irreducibles, each of the form $V^{\left(n-K+c_{i}\right)}\otimes M_{i}$,
$M_{i}$ of degree at most $2K$. Thus, decomposing $M_{i}=\underset{\lambda}{\bigoplus}V^{\lambda}$,
we have $\ell\left(\lambda\right)\leq2K$ for each $\lambda$ involved in the
decomposition. Part \textbf{(1)} of Corollary \ref{der_M} now yields that if  $\ell\left(\mu\right)>2K+1$,
the multiplicity of $V^{\mu}$ in $V^{\left(n-K+c_{i}\right)}\otimes M_{i}$
is zero.

We prove \textbf{(2)} by induction on $n\geq3K$. If $n=3K$, there is nothing to show. Assuming that \textbf{(2)} is true for $n>3K$, by the transitivity of derivations  (Remark \ref{simplest}\textbf{(2)}) it is enough to show that $\left[A_{k}^{\otimes_{s}n+1}\right]^{\left(n+K+1\right)}$
is derived from $\left[A_{k}^{\otimes_{s}n}\right]^{\left(n+K\right)}$.

According to  Lemma \ref{M_lemma}, there are $GL_{k}\left(\mathbb{F}\right)$-representations $M_{1},...,M_{p}\subseteq A_{k}$
of degree at most $2K$, and non-negative integers $c_{1},...,c_{p}$, such
that 
\begin{flalign*}
\left[A_{k}^{\otimes_{s}n}\right]^{\left(n+K\right)} & \cong\stackrel[i=1]{q}{\bigoplus}V^{\left(n-K+c_{i}\right)}\otimes M_{i}\\
\left[A_{k}^{\otimes_{s}n}\right]^{\left(n+K+1\right)} & \cong\stackrel[i=1]{q}{\bigoplus}V^{\left(n-K+c_{i}+1\right)}\otimes M_{i}.
\end{flalign*}

By Remark \ref{simplest}\textbf{(1)}, it is enough to prove that for each
$i$,  $V^{\left(n-K+c_{i}+1\right)}\otimes M_{i}$
is derived from  $V^{\left(n-K+c_{i}\right)}\otimes M_{i}$.
As we assume $n>3K$, we obtain  $n-K+c_{i}\geq n-K>2K$,
and thus we are done by part \textbf{(2)} of Corollary \ref{der_M}.

Finally, 
for \textbf{(3)}, as $\left[A_{k}^{\otimes_{s}n}\right]^{\left(n+K\right)}$
is derived from $\left[A_{k}^{\otimes_{s}3K}\right]^{\left(4K\right)}$ for every $n\geq3K$, we clearly have $\underset{\lambda\vdash n+K}{\sum}l_{n,n+K}^{\lambda}=\underset{\lambda\vdash4K}{\sum}l_{3K,4K}^{\lambda}$
for every $n\geq3K$, implying the existence of such a $D$.

\end{proof}

\section{Stabilization of \texorpdfstring{$W_{n,n+K}$}{Lg} }

In this section we prove our main result, which is Theorem \ref{stab}. We start by describing
the $T$-ideal $\left(x^{n}\right)^{T}\triangleleft\mathbb{F}\left\langle X\right\rangle $
as a vector space. An essential tool for doing this is the \textit{$n$-th
symmetric polynomial}, defined by
\begin{gather*}
P_{n}\left(x_{1},...,x_{n}\right):=\underset{\sigma\in S_{n}}{\sum}x_{\sigma\left(1\right)}\cdot\cdot\cdot x_{\sigma\left(n\right)}\in\mathbb{F}\left\langle X\right\rangle .
\end{gather*}

The following proposition is due to Higman and Hall:

\begin{proposition}[\cite{higman1956conjecture}. See also  {\cite[p.77]{drensky2012polynomial}}]
\label{hig}
As a vector space over $\mathbb{F}$, the $T$-ideal $\left(x^{n}\right)^{T}\triangleleft\mathbb{F}\left\langle X\right\rangle $
is spanned by the polynomials 
\begin{gather*}
P_{n}\left(m_{1},...,m_{n}\right),
\end{gather*}
where $m_{1},...,m_{n}$ are monomials in $\mathbb{F}\left\langle X\right\rangle $.
\end{proposition}

Recall that $A_{k}$ denotes the free algebra $\mathbb{F}\left\langle x_{1},...,x_{k}\right\rangle $.
Using Proposition \ref{hig}, Drensky and Benanti have made the following remarkable observation.

\begin{proposition}
\label{omomorphic}
As a $GL_{k}\left(\mathbb{F}\right)$-representation, $A_{k}\cap\left(x^{n}\right)^{T}$
is a homomorphic image of $A_{k}^{\otimes_{s}n}$, via the map 
\begin{flalign*}
\phi & :A_{k}^{\otimes_{s}n}\rightarrow A_{k}\cap\left(x^{n}\right)^{T}\\
\phi & \left(z_{1}\otimes_{s}\cdot\cdot\cdot\otimes_{s}z_{n}\right)=P_{n}\left(z_{1},...,z_{n}\right).
\end{flalign*}

Similarly, $\left[A_{k}\right]^{\left(m\right)}\cap\left(x^{n}\right)^{T}$ is
a homomorphic image of $\left[A_{k}^{\otimes_{s}n}\right]^{\left(m\right)}$,
the polynomials in $A_{k}^{\otimes_{s}n}$ of  degree $m$. 
\end{proposition}

This proposition allows us to prove one of our key results. 

\begin{theorem}
\label{upper_W}

Fix $K\in\mathbb{N}$, and let $\underset{\lambda\vdash n+K}{\bigoplus}m_{n,n+K}^{\lambda}S^{\lambda}$
be the decomposition of $W_{n,n+K}$ into $S_{n+K}$-irreducible representations.

$\textbf{(1)}$ For every $\lambda\vdash n+K$ such that $\ell\left(\lambda\right)>2K+1$,
$m_{n,n+K}^{\lambda}=0$.

\textbf{(2)} For every $n\geq3K$,  if $m_{n,n+K}^{\lambda}\neq0$
then $\lambda$ is derived (in the sense of Definition \ref{derived}) from some partition of $4K$. 

$\textbf{(3)}$ There exists $D\in\mathbb{N}$ (depending only on $K$),
such that for every $n$ 
\begin{gather*}
\underset{\lambda\vdash n+K}{\sum}m_{n,n+K}^{\lambda}\leq D.
\end{gather*}

\end{theorem}

\begin{proof}
By Theorem \ref{Connection_Sn_Glk}, for every $k$, we have an isomorphism of $GL_{k}\left(\mathbb{F}\right)$-representations
\begin{gather*}
\left[A_{k}\right]^{\left(n+K\right)}\cap\left(x^{n}\right)^{T}\cong\underset{\underset{\ell\left(\lambda\right)\leq k}{\lambda\vdash n+K}}{\bigoplus}m_{n,n+K}^{\lambda}V^{\lambda}.
\end{gather*}

On the other hand, according to Proposition \ref{omomorphic}, $\left[A_{k}\right]^{\left(n+K\right)}\cap\left(x^{n}\right)^{T}$
is a homomorphic image of $\left[A_{k}^{\otimes_{s}n}\right]^{\left(n+K\right)}$,
so \textbf{(1)}  \textbf{(2)} and \textbf{(3)} are inferred directly from Theorem
\ref{pre_main}.
\end{proof}

\subsection{The main result}

\begin{definition}
We say that a sequence $\left(a_{n}\right)_{n\in\mathbb{N}}\subset\mathbb{R}$
is \textbf{generally increasing}, if for every $d$ there exists $N_{d}\in\mathbb{N}$
such that for every $n\geq N_{d}$ we have $a_{d}\leq a_{n}$.
\end{definition}
One can easily verify the following:
\begin{observation}
\label{convergent}
If $\left(a_{n}\right)_{n\in\mathbb{N}}$ is generally increasing and bounded, then it is  convergent.
\end{observation}

Before presenting our main result  (Theorem \ref{stab}), let us give one more theorem.

\begin{theorem}
\label{Tricky}
Fix $K\in\mathbb{N}$ and $\lambda\vdash4K$. There exists $N=N_{K}\in\mathbb{N}$
such that the sequence $\left(m_{n,n+K}^{\lambda^{\left(n-3K\right)}}\right)_{n=N}^{\infty}$
is generally increasing. In other words, 
if $n\geq N$ is large enough, the multiplicity of $S^{\lambda^{\left(n-3K\right)}}$
in $W_{n,n+K}$ is at least the multiplicity of $S^{\lambda^{\left(N-3K\right)}}$
in $W_{N,N+K}$. 

\end{theorem}

We believe, however, that this result can be considerably strengthened.
\begin{conjecture}
In  Theorem \ref{Tricky}, one can take  $N_{K}=1$  and any $n\geq 3K$.
\end{conjecture}

The majority of Section 4  is devoted to proving Theorem \ref{Tricky}. However, assuming Theorem \ref{Tricky}, we can prove now the main result of this paper. 
\begin{theorem}
\label{stab}
Fix $K\in\mathbb{N}$. There exists $N\in\mathbb{N}$ such that for
every $n\geq N$, $W_{n,n+K}$ is derived (as an
$S_{n+K}$-representation) from $W_{N,N+K}$. In other words, if $W_{N,N+K}$ decomposes as \linebreak $W_{N,N+K}\cong\underset{\lambda\vdash N+K}{\bigoplus}m_{N,N+K}^{\lambda}S^{\lambda}$, then for every $n\geq N$ 
\begin{gather*}
W_{n,n+K}\cong\underset{\lambda\vdash N+K}{\bigoplus}m_{N,N+K}^{\lambda}S^{\lambda^{\left(n-N\right)}}.
\end{gather*}
\end{theorem}
\begin{proof}
By part \textbf{(2)} of Theorem \ref{upper_W}, for every $n\geq3K$, if $m_{n,n+K}^{\lambda}\neq0$,
then $\lambda$ is derived from some partition $\mu\vdash4K$. It follows that
for every $n$, we should focus only on such partitions, i.e., partitions from the finite set
\begin{align}
  \left\{ \lambda\vdash n+K\mid\exists\mu\vdash4K\,\,\,\,\,\lambda\text{ is derived from \ensuremath{\mu}}\right\} 
 & =\left\{ \ensuremath{\mu}^{\left(n-3K\right)}:\mu\vdash4K\right\} .
\label{R5}
\end{align}

Let $\mu$ be a partition of $4K$. By part \textbf{(3)} of Theorem \ref{upper_W}, there exists $D\in\mathbb{N}$
such that for every $n$ we have $\underset{\lambda\vdash n+K}{\sum}m_{n,n+K}^{\lambda}\leq D$.
In particular, $m_{n,n+K}^{\mu^{\left(n-3K\right)}}\leq D$. Since
by Theorem \ref{Tricky} the sequence $\left(m_{n,n+K}^{\lambda^{\left(n-3K\right)}}\right)_{n=M_{K}}^{\infty}$
is generally increasing (for some $M_{K}\in\mathbb{N}$), it is convergent
by Observation \ref{convergent}, forcing it to be eventually a constant sequence (being a sequence of integers).
As the set in (\ref{R5}) is finite for every $n$, for  a large enough $n$ all the multiplicities
must stabilize, and the theorem follows.
\end{proof}
\section{Proof of Theorem \ref{Tricky}}
In this section, we prove Theorem \ref{Tricky} in four steps.

In Step 1, we describe the vector space $W_{n,m}$ in terms
of the polynomial $P_{n}$, obtaining a first reduction (Lemma \ref{First_reduction}).
In step 2, we give a further reduction  (Lemma \ref{injection}), where we utilize
Corollary \ref{regev_col} from Section 2.
This reduction is settled by Lemma \ref{tecnical}  in step 3. The proof of Lemma \ref{tecnical} is  somewhat technical and will be presented in step 4.

\subsection{Step 1}

In order to state Proposition \ref{hig} in terms of $W_{n,m}$,
some definitions are required.

\begin{definition}
\label{parpar}
An \textbf{ordered partition} of $\left[m\right]=\left\{ 1,...,m\right\} $
is a set of the form $O=\left\{ A_{1},...,A_{n}\right\} $, where
each $A_{i}$ is an ordered list of numbers, such that (as sets) $A_{i}\cap A_{j}=\emptyset$
for $i\neq j$, and $A_{1}\cup...\cup A_{n}=\left[m\right]$. To avoid confusion with the notation of a permutation, we shall write  $A_{i}=[\![ab...c]\!]$
instead of $A_{i}=\left(a,b,...,c\right)$. The elements $A_{1},...,A_{n}$ are called the \textbf{parts} of $O$. The
set of all ordered partition of $\left[m\right]$ consisting of $n$
parts is denoted by $\Omega_{n,m}$. 

For a given $\lambda=\left(\lambda_{1},\lambda_{2},...,\lambda_{n}\right)\vdash m$,
we denote by $\Omega_{n,m}^{\lambda}$ the subset of $\Omega_{n,m}$ consisting of all the ordered partitions
$O=\left\{ A_{1},...,A_{n}\right\} $ such that 
\begin{gather*}
\left|A_{1}\right|=\lambda_{1},\left|A_{2}\right|=\lambda_{2},...,\left|A_{n}\right|=\lambda_{n}.
\end{gather*}
\end{definition}
\begin{definition}
 Given a list of natural numbers $A=[\![ab...c]\!] $,
we denote by $x_{A}$ the monomial $x_{A}:=x_{a}x_{b}\cdot\cdot\cdot x_{c}\in\mathbb{F}\left\langle x_{1},x_{2},....\right\rangle$.
For every ordered partition $O=\left\{ A_{1},...,A_{n}\right\} \in\Omega_{n,m}$,
we define the following element of $V_{m}$: 
\begin{gather*}
P_{O}:=P_{n}\left(x_{A_{1}},...,x_{A_{n}}\right)=\underset{\sigma\in S_{n}}{\sum}x_{A_{\sigma\left(1\right)}}\cdot\cdot\cdot x_{A_{\sigma\left(n\right)}}.
\end{gather*}
\end{definition}

\begin{example}
\label{Exa3}
The
elements of $\Omega_{2,3}$ are
\begin{gather*}
\left\{ [\![12]\!],[\![3]\!]\right\} ,\left\{ [\![13]\!],[\![2]\!]\right\} ,\left\{ [\![23]\!],[\![1]\!]\right\} \\
\left\{ [\![21]\!],[\![3]\!]\right\} ,\left\{ [\![31]\!],[\![2]\!]\right\} ,\left\{ [\![32]\!],[\![1]\!]\right\} .
\end{gather*}
\end{example}

\begin{example}
If $O=\left\{ [\![12]\!],[\![3]\!]\right\} $, then
$P_{O}=P_{2}\left(x_{1}x_{2},x_{3}\right)=x_{1}x_{2}x_{3}+x_{3}x_{1}x_{2}$.
\end{example}

One should observe  that  $P_{n}\left(m_{1},...,m_{n}\right)\in\mathbb{F}\left\langle X\right\rangle $
lies in $V_{m}$ if and only if $P_{n}\left(m_{1},...,m_{n}\right)=P_{O}$
for some $O\in\Omega_{n,m}$. Hence, Proposition \ref{hig} yields the following,
more explicit description of $W_{n,m}$.

\begin{corollary}
\label{Deacrip}
We have $W_{n,m}=\text{span}_{\mathbb{F}}\left\{ P_{O}:O\in\Omega_{n,m}\right\} $.
\end{corollary}

\begin{example}
Since $O=\left\{ [\![1]\!],[\![2]\!]\right\}$ is the only element in $\Omega_{2,2}$, we obtain Example \ref{once_ag} directly:
\begin{gather*}
W_{2,2}=\text{span}_{\mathbb{F}}\left\{ P_{O}:O\in\Omega_{2,2}\right\} =\text{span}_{\mathbb{F}}\left\{ P_{2}\left(x_{1},x_{2}\right)\right\} =\text{span}_{\mathbb{F}}\left\{ x_{1}x_{2}+x_{2}x_{1}\right\} .
\end{gather*}
\end{example}
\begin{example}
By Example \ref{Exa3}, we obtain that $W_{2,3}$ is spanned
by the following six elements: 
\begin{gather*}
P_{2}\left(x_{1}x_{2},x_{3}\right),P_{2}\left(x_{1}x_{3},x_{2}\right),P_{2}\left(x_{2}x_{3},x_{1}\right)\\
P_{2}\left(x_{2}x_{1},x_{3}\right),P_{2}\left(x_{3}x_{1},x_{2}\right),P_{2}\left(x_{3}x_{2},x_{1}\right).
\end{gather*}
\end{example}

\begin{remark}
Note that $S_{m}$ acts naturally on $\Omega_{n,m}$ via 
\begin{gather*}
\sigma.\left\{ [\![a...b]\!],[\![c...d]\!],....\right\} =\left\{ \left[\sigma\left(a\right)...\sigma\left(b\right)\right],\left[\sigma\left(c\right)...\sigma\left(d\right)\right],....\right\}, 
\end{gather*}

and one can verify that the usual action of $S_{m}$ on
$W_{n,m}$ is induced by this action.
In terms of representation theory, this translates into saying that $W_{n,m}$
 is a homomorphic image of the permutation representation $\mathbb{F}\left[\Omega_{n,m}\right]$,
via the $\mathbb{F}\left[S_{m}\right]$-homomorphism 
\begin{flalign*}
\psi & :\mathbb{F}\left[\Omega_{n,m}\right]\rightarrow W_{n,m}\\
\psi & \left(O\right)=P_{O}.
\end{flalign*}

Hence, it would seem natural to use the characters of $\mathbb{F}\left[\Omega_{n,m}\right]$
 for obtaining some information about the characters of $W_{n,m}$.
However, we are not aware of any "nice" description of the characters of $\mathbb{F}\left[\Omega_{n,m}\right]$.
In fact, using the notations of Definition \ref{parpar}, the well-known Foulkes's conjecture
(see, e.g., \cite{dent2000conjecture} or {\cite[p.227]{james2006representation}}) concerns the decomposition of $\mathbb{F}\left[\Omega_{a,ab}^{\left(b,...,b\right)}\right]$.
\end{remark}

\begin{definition}
Given an ordered set $O=\left\{ A_{1},...,A_{n}\right\} \in\Omega_{n,m}$,
we denote by $O^{\left(s\right)}\in\Omega_{n+s,m+s}$ the ordered
set 
\begin{gather*}
O^{\left(s\right)}=\left\{ A_{1},...,A_{n},[\![m+1]\!],...,[\![m+s]\!]\right\} .
\end{gather*}
Note that if $P_{O}=P_{O}\left(x_{1},...,x_{m}\right)\in W_{n,m}$, 
then $P_{O^{\left(s\right)}}=P_{O^{\left(s\right)}}\left(x_{1},...,x_{m},x_{m+1},...,x_{m+s}\right)\in W_{n+s,m+s}$.
\end{definition}
\begin{example}
If $O=\left\{ [\![12]\!],[\![3]\!]\right\} $, then $O^{\left(2\right)}=\left\{ [\![12]\!],[\![3]\!],[\![4]\!],[\![5]\!]\right\} $,
and 
\begin{flalign*}
P_{O} & =P_{2}\left(x_{1}x_{2},x_{3}\right)\\
P_{O^{\left(2\right)}} & =P_{4}\left(x_{1}x_{2},x_{3},x_{4},x_{5}\right).
\end{flalign*}
\end{example}

\begin{definition}
Let $\lambda=\left(a_{1},...,a_{r}\right)\vdash m$ and $T\in\text{Tab}\left(\lambda\right)$.
We denote by $T^{\left(s\right)}\in\text{Tab}\left(\lambda^{\left(s\right)}\right)$
the tableau obtained from $T$ by adjoining the boxes  ${  \begin{ytableau} \scriptstyle m+1 & ... & \scriptstyle m+k   \end{ytableau}} $.

\end{definition}

\begin{example}
 If $  T={\tiny \ \begin{ytableau} 1 & 2  \\ 3   \end{ytableau}} \in\text{Tab}\left(\left(2,1\right)\right)$,
then $  T^{\left(3\right)}={\tiny \ \begin{ytableau} 1 & 2 & 4 & 5 & 6  \\ 3   \end{ytableau}}\in\text{Tab}\left(\left(5,1\right)\right)$.
\end{example}

We are now ready to make a first reduction towards proving Theorem \ref{Tricky}. 

\begin{lemma}
\label{First_reduction}
Fix $K\in\mathbb{N}$ and $\lambda\vdash4K$. There exists an integer
$3K\leq M_{K}=N$ such that, for every $T\in\text{Tab}\left(\lambda^{\left(N-3K\right)}\right)$,
if $e_{T}P_{O_{1}},...,e_{T}P_{O_{l}}$ are linearly independent for
some $P_{O_{1}},...,P_{O_{l}}\in W_{N,N+K}$, then for all $n\geq N$
large enough,
\begin{gather*}
e_{T^{\left(n-N\right)}}P_{O_{1}^{\left(n-N\right)}},...,e_{T^{\left(n-N\right)}}P_{O_{l}^{\left(n-N\right)}}\in W_{n,n+K}
\end{gather*}
are also linearly independent.
\end{lemma}

\begin{claim}
\label{claim}
Lemma \ref{First_reduction} implies Theorem \ref{Tricky}.
\end{claim}

\begin{proof}
Fix $K\in\mathbb{N}$ and $\lambda\vdash4K$, and let $M_{K}=N$ be
the integer provided by Lemma \ref{First_reduction}. Suppose that $l=m_{N,N+K}^{\lambda^{\left(N-3K\right)}}$ is the multiplicity of
$S^{\lambda^{\left(N-3K\right)}}$ in $W_{N,N+K}$. Lemma $\ref{mult_is_dim}$
tells us that for any $T\in\text{Tab}\left(\lambda^{\left(N-3K\right)}\right)$,
we have $l=\text{dim}\left(e_{T}W_{N,N+K}\right)$.
Using the description of $W_{n,m}$ from Corollary \ref{Deacrip}, we obtain the existence of $l$
ordered partitions $O_{1},...,O_{l}\in\Omega_{N,N+K}$ and \linebreak $T\in\text{Tab}\left(\lambda^{\left(N-3K\right)}\right)$
such that 
\begin{gather*}
e_{T}P_{O_{1}},...,e_{T}P_{O_{l}}
\end{gather*}
are linearly independent. Assuming Lemma \ref{First_reduction}, we obtain that for all $n\geq N$
large enough  
\begin{gather*}
e_{T^{\left(n-N\right)}}P_{O_{1}^{\left(n-N\right)}},...,e_{T^{\left(n-N\right)}}P_{O_{l}^{\left(n-N\right)}}\in e_{T^{\left(n-N\right)}}W_{n,n+K}
\end{gather*}
are also linearly independent. Note that $T^{\left(n-N\right)}\in\text{Tab}\left(\lambda^{\left(n-N\right)}\right)$,
and thus (using Corollary \ref{Deacrip} and Lemma $\ref{mult_is_dim}$ once again), the multiplicity of $S^{\lambda^{\left(n-N\right)}}$ in $W_{n,n+K}$
is at least $l$.

\end{proof}

\subsection{Step 2 - Proof of Lemma \ref{First_reduction}}

To prove Lemma \ref{First_reduction}, let us make another reduction.
Recall from Definition \ref{Regevsub} that the substitution $\text{Sub}_{\left(\text{x},\text{y}\right)}:\mathbb{F}\left\langle x\right\rangle \rightarrow\mathbb{F}\left\langle Y\right\rangle $
is defined by $x_{j}\mapsto y_{j}$, where we set $f\left(\text{y}\right)=\text{Sub}_{\left(\text{x},\text{y}\right)}\left(f\left(\text{x}\right)\right)$
for each $f\in\mathbb{F}\left\langle X\right\rangle $. We also remind the reader
that any tableau $T$ induces a substitution
map $\text{Sub}_{\left(\text{x},\text{y}\right)}^{T}:x_{j}\mapsto y_{i_{j}}$,
where $i_{j}$ is the number of the row for which $j$ appears in $T$. 

\begin{definition}
Let $P=P_{n}\left(m_{1},...,m_{n}\right)$, where $\left\{ m_{i}\right\} $
are monomials in $\mathbb{F}\left\langle Y\right\rangle $. Given $s\in\mathbb{N}$, we define
\begin{gather*}
P^{\left(s\right)}=P_{n+s}(m_{1},...,m_{n},\stackrel{s\text{ times}}{\overbrace{y_{1},...,y_{1}}}).
\end{gather*}
We extend this notation linearly: if $P_{1}=P_{n}\left(m_{11},...,m_{n1}\right),...,P_{r}=P_{n}\left(m_{1r},...,m_{nr}\right)$,
then for a linear combination $f=\stackrel[i=1]{r}{\sum}\alpha_{i}P_{i}$,
we set $f^{\left(s\right)}:=\stackrel[i=1]{r}{\sum}\alpha_{i}P_{i}^{\left(s\right)}$.
\end{definition}

\begin{remark}
Note that the map sending $f\mapsto f^{\left(s\right)}$
is usually not linear. For instance, taking 
\begin{gather*}
f_{1}=P_{2}\left(y_{1}y_{1},y_{2}y_{2}\right),f_{2}=P_{2}\left(y_{1}y_{2},y_{2}y_{1}\right)\\
g_{1}=P_{2}\left(y_{1},y_{1}y_{2}y_{2}\right),g_{2}=P_{2}\left(y_{2},y_{2}y_{1}y_{1}\right),
\end{gather*}

then $f_{1}+f_{2}=y_{1}y_{1}y_{2}y_{2}+y_{2}y_{2}y_{1}y_{1}+y_{1}y_{2}y_{2}y_{1}+y_{2}y_{1}y_{1}y_{2}=g_{1}+g_{2}$,
but one may verify that $f_{1}^{\left(1\right)}+f_{2}^{\left(1\right)}\neq g_{1}^{\left(1\right)}+g_{2}^{\left(1\right)}$.
However, in case $f_{1},...,f_{r}\in\mathbb{F}\left\langle Y\right\rangle $
are linearly independent, \linebreak $\left(\cdot\right)^{\left(s\right)}:\text{span}_{\mathbb{F}}\left\{ f_{1},...,f_{r}\right\} \rightarrow\mathbb{F}\left\langle Y\right\rangle $
 is indeed a well-defined linear map.
\end{remark}

\begin{observation}
\label{expes}
Let $P_{O}=P_{O}\left(x_{1},...,x_{m}\right)\in W_{n,m}$, and let $T$ be any tableau of shape $\lambda\vdash m$. For every
$s\in\mathbb{N}$, it is clear that the columns of length $\geq2$ of $T$  and  $T^{\left(s\right)}$ coincide,
and thus $C_{T}=C_{T^{\left(s\right)}}$ (see Definition \ref{elemen}).

Writing $P_{O^{\left(s\right)}}=P_{O^{\left(s\right)}}\left(x_{1},...,x_{m},x_{m+1},...,x_{m+s}\right)\in V_{m+s}$,
and identifying $V_{m+s}=\mathbb{F}\left[S_{m+s}\right]$ as usual,
we obtain 
\begin{flalign*}
b_{T^{\left(s\right)}}\left(\text{y}\right)P_{O^{\left(s\right)}} & \stackrel{\left(\dagger\right)}{=}\underset{\tau\in C_{T^{\left(s\right)}}}{\sum}\text{sgn}\left(\tau\right)P_{O^{\left(s\right)}}\left(y_{i_{\tau\left(1\right)}},...,y_{i_{\tau\left(m\right)}},y_{i_{\tau\left(m+1\right)}},...,y_{i_{\tau\left(m+s\right)}}\right)\\
 & =\underset{\tau\in C_{T}}{\sum}\text{sgn}\left(\tau\right)P_{O^{\left(s\right)}}\left(y_{i_{\tau\left(1\right)}},...,y_{i_{\tau\left(m\right)}},y_{i_{m+1}},...,y_{i_{m+s}}\right)\\
 & \stackrel{\left(\dagger\dagger\right)}{=}\underset{\tau\in C_{T}}{\sum}\text{sgn}\left(\tau\right)P_{O^{\left(s\right)}}\left(y_{i_{\tau\left(1\right)}},...,y_{i_{\tau\left(m\right)}},y_{1},...,y_{1}\right)\\
 & =\underset{\tau\in C_{T}}{\sum}\text{sgn}\left(\tau\right)P_{O}^{\left(s\right)}\left(y_{i_{\tau\left(1\right)}},...,y_{i_{\tau\left(m\right)}}\right),
\end{flalign*}

where the equality $\left(\dagger\right)$ holds by Corollary \ref{inter},
and the equality $\left(\dagger\dagger\right)$ holds since (by construction) the first
row of $T^{\left(s\right)}$ contains the numbers $m+1,...,m+s$. 
\end{observation}

\begin{lemma}
\label{injection}
Fix $K\in\mathbb{N}$, and let $\mu\vdash4K$. There exists an integer
$3K\leq M_{K}$ such that, for every $N\geq M_{K}$, $T\in\text{Tab}\left(\mu^{\left(N-3K\right)}\right)$ and
$O_{1},...,O_{l}\in\Omega_{N,N+K}$, if
\begin{gather*}
f_{1},...,f_{r}\in\text{span}_{\mathbb{F}}\left\{ \text{Sub}_{\left(\text{x},\text{y}\right)}^{T}\left(P_{O_{1}}\right),...,\text{Sub}_{\left(\text{x},\text{y}\right)}^{T}\left(P_{O_{l}}\right)\right\} 
\end{gather*}

are linearly independent, then for all large enough $s$, the map $\left(\cdot\right)^{\left(s\right)}:\text{span}_{\mathbb{F}}\left\{ f_{1},...,f_{r}\right\} \rightarrow\mathbb{F}\left\langle Y\right\rangle $
is injective.
\end{lemma}

\begin{claim}
Lemma \ref{injection}  implies Lemma \ref{First_reduction}.
\end{claim}

\begin{proof}
Let $\lambda\vdash4K$. Choose $N$ large enough for which
Lemma \ref{injection} holds, and take any \linebreak $T\in\text{Tab}\left(\lambda^{\left(N-3K\right)}\right)$.
Assuming that $e_{T}P_{O_{1}},...,e_{T}P_{O_{l}}\in W_{N,N+K}$ are
linearly independent,  then by applying Corollary \ref{regev_col} we obtain that 
\begin{gather*}
\underset{\tau\in C_{T}}{\sum}\text{sgn}\left(\tau\right)P_{O_{1}}\left(y_{i_{\tau\left(1\right)}},...,y_{i_{\tau\left(N+K\right)}}\right),...,\underset{\tau\in C_{T}}{\sum}\text{sgn}\left(\tau\right)P_{O_{l}}\left(y_{i_{\tau\left(1\right)}},...,y_{i_{\tau\left(N+K\right)}}\right)
\end{gather*}
are linearly independent. According to Lemma \ref{injection},  for sufficiently large
 $s$
\begin{gather*}
\underset{\tau\in C_{T}}{\sum}\text{sgn}\left(\tau\right)P_{O_{1}}^{\left(s\right)}\left(y_{i_{\tau\left(1\right)}},...,y_{i_{\tau\left(N+K\right)}}\right),...,\underset{\tau\in C_{T}}{\sum}\text{sgn}\left(\tau\right)P_{O_{l}}^{\left(s\right)}\left(y_{i_{\tau\left(1\right)}},...,y_{i_{\tau\left(N+K\right)}}\right)
\end{gather*}
are linearly independent, which by Observation \ref{expes} means that for sufficiently large
 $s$
\begin{gather*}
b_{T^{\left(s\right)}}\left(\text{y}\right)P_{O_{1}^{\left(s\right)}},...,b_{T^{\left(s\right)}}\left(\text{y}\right)P_{O_{l}^{\left(s\right)}}
\end{gather*}
are linearly independent. Applying Corollary \ref{regev_col} once again, we obtain
that for sufficiently large $s$,
\begin{gather*}
e_{T^{\left(s\right)}}P_{O_{1}^{\left(s\right)}},...,e_{T^{\left(s\right)}}P_{O_{l}^{\left(s\right)}}
\end{gather*} are also linearly independent, and Lemma \ref{First_reduction} follows.
\end{proof}

\subsection{Step 3 - Proof of Lemma \ref{injection}}

Here we start to prove Lemma \ref{injection}. For this purpose, one needs
 to establish a connection between the coefficients of $P\left(y_{i_{1}},...,y_{i_{m}}\right)=P_{n}\left(m_{1},...,m_{n}\right)$
and the coefficients of $P^{\left(s\right)}\left(y_{i_{1}},...,y_{i_{m}}\right)=P_{n+s}\left(m_{1},...,m_{n},y_{1},...,y_{1}\right)$.

Let us denote $\text{Coef}_{u}\left(P\right):=$ the coefficient of
the monomial $u\in\mathbb{F}\left\langle Y\right\rangle $ in $P$.

\begin{definition}
Let $u\in\mathbb{F}\left\langle Y\right\rangle $ be a monomial.

We say that $v\in\mathbb{F}\left\langle Y\right\rangle $ is a \textbf{submonomial} of $u$, and write $v\subset u$,
if $u$ can be  written as $u=gvh$ for some (possibly empty) monomials $g,h\in\mathbb{F}\left\langle Y\right\rangle $.

We say that $u$ has \textbf{length} $m$, and write $\left|u\right|=m$,
if $u=y_{i_{1}}\cdot\cdot\cdot y_{i_{m}}$ for some variables $\left\{ y_{i_{j}}\right\} \subset\mathbb{F}\left\langle Y\right\rangle $.

The \textbf{central part} of $u$, denoted by $\text{C}\left(u\right)$,
is the largest submonomial of $u$ of the form $y_{1}^{r}$, (if it
exists). In case of more than one such submonomial, we take the
leftmost one.
\end{definition}

\begin{definition}
\label{centrall}
Let $u\in\mathbb{F}\left\langle Y\right\rangle $ be a monomial for which $\text{C}\left(u\right)$
exists, and write $u=g\cdot\text{C}\left(u\right)\cdot h$ for some
(possibly empty) monomials $g,h\in\mathbb{F}\left\langle Y\right\rangle $.
We define for any $s\in\mathbb{Z}_{\geq0}$ 
\begin{gather*}
u^{\left(s\right)}=g\cdot y_{1}^{s}\text{C}\left(u\right)\cdot h.
\end{gather*}

\end{definition}

\begin{remark}
\label{Cs}
It is clear that for every $s$, $\left|\text{C}\left(u^{\left(s\right)}\right)\right|=\left|\text{C}\left(u\right)\right|+s$.
\end{remark}

The following lemma is the key to proving Lemma \ref{injection}.

\begin{lemma}
\label{tecnical}
Fix $K\in\mathbb{N}$, and let $\mu\vdash4K$. There exists an integer
$M_{K}\geq3K$ such that for every $n\geq M_{K}$, the following holds:
if $T\in\text{Tab}\left(\mu^{\left(n-3K\right)}\right)$, $O\in\Omega_{n,n+K}$,
and we set $P:=\text{Sub}_{\left(\text{x},\text{y}\right)}^{T}\left(P_{O}\right)$,
then for every monomial $u$ in $P$, there exists a polynomial
$p\left(x\right)\in\mathbb{F}\left[x\right]$ such that 
\begin{gather*}
\text{Coef}_{u^{\left(s\right)}}\left(P^{\left(s\right)}\right)=p\left(s\right)\cdot s!
\end{gather*}
for every $s\in\mathbb{Z}_{\geq0}$.
\end{lemma}

Before presenting the proof of Lemma \ref{tecnical}, let us show how Lemma \ref{injection} follows from it.

\begin{proof}[Proof of Lemma \ref{injection}]
Suppose $n$ is large enough so that Lemma \ref{tecnical} holds, take \linebreak $O_{1},...,O_{l}\in\Omega_{n,n+K}$, and set
\begin{gather*}
P_{i}:=\text{Sub}_{\left(\text{x},\text{y}\right)}^{T}\left(P_{O_{i}}\right)\,\,\,\,\,\,(i\in\left\{ 1,...,l\right\} ).
\end{gather*}

Assuming that $f_{1},...,f_{r}\in\text{span}_{\mathbb{F}}\left\{ P_{1},...,P_{l}\right\} $
are linearly independent, we can find $r$ monomials $u_{1},...,u_{r}$
such that the matrix 
\begin{gather*}
\left[\text{Coef}_{u_{i}}\left(f_{j}\right)\right]_{\left(i,j\right)\in\left[r\right]^{2}}
\end{gather*}
has a nonzero determinant. By  Lemma \ref{tecnical}, there are polynomials
$p_{1,1},...,p_{r,l}$ such that $\frac{1}{s!}\text{Coef}_{u_{i}^{\left(s\right)}}\left(P_{j}^{\left(s\right)}\right)=p_{i,j}\left(s\right)$
for every $s\in\mathbb{Z}_{\geq0}$. As $f_{1},...,f_{r}$
are linear combination of $P_{1},...,P_{l}$, there exist polynomials
$\left\{ q_{i,j}\right\} _{\left(i,j\right)\in\left[r\right]^{2}}$
such that for every $s\in\mathbb{Z}_{\geq0}$ we have
\begin{gather*}
\frac{1}{s!}\text{Coef}_{u_{i}^{\left(s\right)}}\left(f_{j}^{\left(s\right)}\right)=q_{i,j}\left(s\right).
\end{gather*}

It follows that 
\begin{gather*}
F\left(s\right):=\text{det}\left(\left[\frac{1}{s!}\text{Coef}_{u_{i}^{\left(s\right)}}\left(f_{j}^{\left(s\right)}\right)\right]_{\left(i,j\right)\in\left[r\right]^{2}}\right)=\text{det}\left(\left[q_{i,j}\left(s\right)\right]_{\left(i,j\right)\in\left[r\right]^{2}}\right)
\end{gather*}
is a polynomial in $s$. Since we assume $F\left(0\right)\neq0$,
$F$ is a nonzero polynomial, and thus has only a finite number of
roots. In particular, for every large enough $s$, we have $F\left(s\right)\neq0$,
implying that $f_{1}^{\left(s\right)},...,f_{r}^{\left(s\right)}$
are linearly independent, and $\left(\cdot\right)^{\left(s\right)}:\text{span}_{\mathbb{F}}\left\{ f_{1},...,f_{r}\right\} \rightarrow\mathbb{F}\left\langle Y\right\rangle $
is injective.
\end{proof}

\subsection{Step 4 - Proof of Lemma \ref{tecnical}}

\begin{definition}
Given monomials $m_{1},...,m_{r}\in\mathbb{F}\left\langle Y\right\rangle $ and $s\in\mathbb{Z}_{\geq0}$,
we define 
\begin{gather*}
F_{s}\left(m_{1},...,m_{r}\right):=\underset{\underset{a_{0}+...+a_{r}=s}{a_{0},...,a_{r}\in\mathbb{Z}_{\geq0}}}{\sum}y_{1}^{a_{0}}m_{1}y_{1}^{a_{1}}m_{2}y_{1}^{a_{2}}\cdot\cdot\cdot y_{1}^{a_{r-1}}m_{r}y_{1}^{a_{r}}\in\mathbb{F}\left\langle Y\right\rangle .
\end{gather*}
\end{definition}

 One can verify the following simple proposition:

\begin{proposition}
\label{NNM}
Let $m_{1},...,m_{r}\in\mathbb{F}\left\langle Y\right\rangle $ be
monomials,  and let $P=P_{r}\left(m_{1},...,m_{r}\right)\in\mathbb{F}\left\langle Y\right\rangle $. Then for any
$s\in\mathbb{Z}_{\geq0}$
\begin{gather*}
P^{\left(s\right)}=P_{r+s}(m_{1},...,m_{r},\stackrel{s\text{ times}}{\overbrace{y_{1},...,y_{1}}})=s!\underset{\sigma\in S_{r}}{\sum}F_{s}\left(m_{\sigma\left(1\right)},...,m_{\sigma\left(r\right)}\right).
\end{gather*}
\end{proposition}

\begin{lemma}
\label{tec_oftec}
Fix $q,r\in\mathbb{N}$, and let $m_{1},...,m_{r}\in\mathbb{F}\left\langle Y\right\rangle $
be monomials such that $\left|m_{i}\right|\leq q$ for each $i$.
Then, for any monomial $u\in\mathbb{F}\left\langle Y\right\rangle $
such that $\left|\text{C}\left(u\right)\right|>qr$, the coefficient
of $u^{\left(s\right)}$ in $F_{s}\left(m_{1},...,m_{r}\right)$ is
a polynomial in $s$.
\end{lemma}

Before presenting the proof, we illustrate this lemma with an example.

\begin{example}
Take $m_{1}=y_{2}$, $m_{2}=y_{1}^{2}$ and $u=y_{1}y_{2}y_{1}^{5}$.
Then $\left|m_{i}\right|\leq2$, and  $\left|\text{C}\left(u\right)\right|=\left|y_{1}^{5}\right|=5>4$.
Note that for every $n\neq$ $6$, the coefficient of $u^{\left(s\right)}=y_{1}y_{2}y_{1}^{5+s}$
in 
\begin{gather*}
F_{s+n}\left(y_{2},y_{1}^{2}\right)=\underset{\underset{a_{0}+a_{1}+a_{2}=s+n}{a_{0},a_{1},a_{2}\in\mathbb{Z}_{\geq0}}}{\sum}y_{1}^{a_{0}}\left(y_{2}\right)y_{1}^{a_{1}}\left(y_{1}^{2}\right)y_{1}^{a_{2}}=\underset{\underset{a_{0}+a_{1}+a_{2}=s+n}{a_{0},a_{1},a_{2}\in\mathbb{Z}_{\geq0}}}{\sum}y_{1}^{a_{0}}y_{2}y_{1}^{2+a_{1}+a_{2}}
\end{gather*}

equals zero, and when $n=6$, the coefficient of $u^{\left(s\right)}$
in $F_{s+6}\left(y_{2},y_{1}^{2}\right)$ equals 
\begin{flalign*}
\left|\left\{ \left(a_{0},a_{1},a_{2}\right)\in\mathbb{Z}_{\geq0}:a_{0}=1,2+a_{1}+a_{2}=5+s\right\} \right| & =\left|\left\{ \left(a_{1},a_{2}\right)\in\mathbb{Z}_{\geq0}:a_{1}+a_{2}=5+s\right\} \right|\\
 & =6+s.
\end{flalign*}

On the other hand, taking $m_{1}=y_{2}y_{1}y_{2}$, $m_{2}=y_{2}$ and
$u=y_{2}y_{1}y_{2}^{2}y_{1}$, the coefficient of $u^{\left(s\right)}$
in $F_{s+1}\left(y_{2}y_{1}y_{2},y_{2}\right)$ equals 1 for $s=0$,
and 0 for any $s>0$ (since $u^{\left(s\right)}=y_{2}y_{1}^{s+1}y_{2}^{2}y_{1}$
is not of the form $y_{1}^{a_{0}}y_{2}y_{1}y_{2}y_{1}^{a_{1}}y_{2}y_{1}^{a_{2}}$if
$a_{0}+a_{1}+a_{2}>1$). Thus, the coefficient of $u^{\left(s\right)}$
is not a polynomial in $s$.
\end{example}

\begin{proof}
Suppose $u\in\mathbb{F}\left\langle Y\right\rangle $ is a monomial
with $\left|\text{C}\left(u\right)\right|>qr$. Fix some $n\in\mathbb{Z}_{\geq0}$.
If $u^{\left(s\right)}$ is not a monomial in $F_{s+n}\left(m_{1},...,m_{r}\right)$
for every $s\in\mathbb{Z}_{\geq0}$, there is nothing to show (we can take
the zero polynomial). Hence, we may assume that there exist $s,a_{0},...,a_{r}\in\mathbb{Z}_{\geq0}$ satisfying $a_{0}+...+a_{r}=n+s$,
such that 
\begin{gather*}
u^{\left(s\right)}=y_{1}^{a_{0}}m_{1}y_{1}^{a_{1}}m_{2}y_{1}^{a_{2}}\cdot\cdot\cdot y_{1}^{a_{r-1}}m_{r}y_{1}^{a_{r}}.
\end{gather*}
Since we assume $\left|\text{C}\left(u\right)\right|>qr$, we have
$\left|\text{C}\left(u^{\left(s\right)}\right)\right|\stackrel{\ref{Cs}}{=}\left|\text{C}\left(u\right)\right|+s>qr+s$.
Thus, as \linebreak $\left|m_{1}\right|+...+\left|m_{r}\right|\leq qr$, we can
write
\begin{flalign*}
\text{C}\left(u^{\left(s\right)}\right) & =y_{1}^{a_{k}}m_{k+1}\cdot\cdot\cdot m_{t}y_{1}^{a_{t}},
\end{flalign*}

where the numbers $a_{k},...,a_{t}$ satisfy $a_{k}+...+a_{t}>s$,
and
\begin{gather*}
a_{k}+...+a_{t}+\left|m_{k+1}\right|+..+\left|m_{t}\right|=\left|\text{C}\left(u^{\left(s\right)}\right)\right|.
\end{gather*}
Writing 
\begin{flalign*}
A & =y_{1}^{a_{0}}m_{1}\cdot\cdot m_{k}\\
B & =m_{t+1}\cdot\cdot\cdot m_{r}y_{1}^{a_{r}},
\end{flalign*}

we have
\begin{gather*}
u^{\left(s\right)}=A\cdot\text{C}\left(u^{\left(s\right)}\right)\cdot B=A\cdot y_{1}^{a_{k}}m_{k+1}\cdot\cdot\cdot m_{t}y_{1}^{a_{t}}\cdot B.
\end{gather*}

It follows that for every $0\leq s'\leq s$,  $u^{\left(s'\right)}$
can also be written as
\begin{align}
  u^{\left(s'\right)}=A\cdot\text{C}\left(u^{\left(s'\right)}\right)\cdot B=A\cdot y_{1}^{b_{k}}m_{k+1}\cdot\cdot\cdot m_{t}y_{1}^{b_{t}}\cdot B
  \label{p}
\end{align}
for some $b_{k},...,b_{t}$ with $b_{k}+...+b_{t}>s'$. As the right-hand
side of (\ref{p}) is a monomial in 
\begin{gather*}
F_{n+s'}\left(m_{1},...,m_{r}\right)=\underset{\underset{c_{0}+...+c_{r}=n+s'}{c_{0},...,c_{r}\in\mathbb{Z}_{\geq0}}}{\sum}y_{1}^{c_{0}}m_{1}y_{1}^{c_{1}}\cdot\cdot\cdot y_{1}^{c_{r-1}}m_{r}y_{1}^{c_{r}},
\end{gather*}

we obtain that $u^{\left(s'\right)}$ is also a monomial in $F_{n+s'}\left(m_{1},...,m_{r}\right)$,
arising by the same monomials $m_{k+1},...,m_{t}$. Clearly, the same argument works for every $s\leq s'$. Letting $Q:=\left|m_{k+1}\right|+...+\left|m_{t}\right|$,
we see that for \textit{every} $s$, the coefficient of $u^{\left(s\right)}$
in $F_{s+n}\left(m_{1},...,m_{r}\right)$ depends only on the number
\begin{flalign*}
q\left(s\right) & =\left|\left\{ \left(c_{k},...,c_{t}\right):c_{k}+...+c_{t}+\left|m_{k+1}\right|+..+\left|m_{t}\right|=\left|\text{C}\left(u^{\left(s\right)}\right)\right|\right\} \right|\\
 & =\left|\left\{ \left(c_{k},...,c_{t}\right):c_{k}+...+c_{t}+Q=\left|\text{C}\left(u\right)\right|+s\right\} \right|\\
 & =\left(\begin{array}{c}
\left(t-k\right)+\left|\text{C}\left(u\right)\right|-Q+s\\
\left(t-k\right)
\end{array}\right),
\end{flalign*}

which is a polynomial in $s$. 
\end{proof}

\begin{proposition}
\label{6_parts}
Fix $K\in\mathbb{N}$, and let $O=\left\{ A_{1},...,A_{n}\right\} \in\Omega_{n,n+K}$
be an ordered partition of $n+K$. Then

\textbf{(1)} For every $i$,  $\left|A_{i}\right|\leq K+1$.

\textbf{(2)} $\left|\left\{ i\in\left[n\right]:\left|A_{i}\right|\geq2\right\} \right|\leq K$.

\textbf{(3)} Fix $\mu\vdash4K$, and let $n\geq3K$. For every
$T\in\text{Tab}\left(\mu^{\left(n-3K\right)}\right)$, the monomials
$m_{1},...,m_{n} $ in 
\begin{gather*}
P_{n}\left(m_{1},...,m_{n}\right):=\text{Sub}_{\left(\text{x},\text{y}\right)}^{T}\left(P_{O}\right)
\end{gather*}

satisfy the following:

(i) For every $i$, $\left|m_{i}\right|\leq K+1$.

(ii) $\left|\left\{ i\in\left[n\right]:\left|m_{i}\right|\geq2\text{ }\right\} \right|\leq K$.

(iii) $\left|\left\{ i\in\left[n\right]:m_{i}\neq y_{1}\right\} \right|\leq16K^{2}$.
\end{proposition}

\begin{proof}
\textbf{(1)} Considering $A_{1},...,A_{n}$ as sets, we have $A_{1}\sqcup...\sqcup A_{n}=\left[n+K\right]$.
If (without loss of generality) $\left|A_{1}\right|>K+1$, we obtain
a contradiction:
\begin{gather*}
n+K=\left|A_{1}\right|+\left|A_{2}\right|+...+\left|A_{n}\right|>K+1+\stackrel{n-1}{\overbrace{1+\cdot\cdot\cdot+1}}=n+K.
\end{gather*}

Similarly, if $r:=\left|\left\{ i\in\left[n\right]:\left|A_{i}\right|\geq2\right\} \right|$
and without loss of generality $\left|A_{1}\right|,...,\left|A_{r}\right|\geq2$,
then 
\begin{gather*}
n+K=\left|A_{1}\right|+\cdot\cdot\cdot+\left|A_{r}\right|+\left|A_{r+1}\right|+\cdot\cdot\cdot+\left|A_{n}\right|\geq2r+\stackrel{n-r}{\overbrace{1+\cdot\cdot\cdot+1}}=n+r,
\end{gather*}

and \textbf{(2)} follows. For \textbf{(3)}, note that by definition we have 
\begin{gather*}
P_{n}\left(m_{1},...,m_{n}\right)=\text{Sub}_{\left(\text{x},\text{y}\right)}^{T}\left(P_{O}\right)=\text{Sub}_{\left(\text{x},\text{y}\right)}^{T}\left(P_{n}(x_{A_{1}},...,x_{A_{n}})\right),
\end{gather*}
where $m_{i}=\text{Sub}_{\left(\text{x},\text{y}\right)}^{T}(x_{A_{i}})$.
Thus, part (i) follows immediately from \textbf{(1)}. Similarly, part (ii) follows
from \textbf{(2)}.

Finally, for (iii), note that as $\mu$ is a partition of $4K$, its
diagram is contained in a the diagram of size $4K\times4K$. By definition,
the diagram corresponding to $\mu^{\left(n-3K\right)}$
is obtained by adding $\left(n-3K\right)$ boxes to the right upper
corner of the diagram of $\mu$. Since the map $\text{Sub}_{\left(\text{x},\text{y}\right)}^{T}$
substitutes $x_{j}\mapsto y_{i_{j}}$, where $i_{j}$ is the number
of the row for which $j$ appears in $T$, there are at most $4K\left(4K-1\right)<16K^{2}$
variables $x_{j}$ such that $i_{j}\neq1$ and $x_{j}\mapsto y_{i_{j}}$.
\end{proof}

\begin{proposition}
\label{d_large}
Let $P=P_{n}\left(m_{1},...,m_{n}\right)\in\mathbb{F}\left\langle Y\right\rangle $
be a symmetric polynomial, and let $p>0$.
If $d$ is large enough, then   $\left|\text{C}\left(u\right)\right|\geq p$ for every monomial in $P^{\left(d\right)}$.
\end{proposition}

\begin{proof}
 By definition, the monomials in $P^{\left(d\right)}=P_{n}\left(m_{1},...,m_{n},y_{1},...,y_{1}\right)$
are obtained by placing the variables $\stackrel{d}{\overbrace{y_{1},...,y_{1}}}$
between the monomials $m_{1},...,m_{n}$. Thus,  when $d$ tends to
infinity, each monomial in $P^{\left(d\right)}$ must contain some block
of the form $y_{1}\cdot\cdot\cdot\cdot y_{1}$ of size at least $p$ by the pigeonhole principle.
\end{proof}
We are now ready to prove Lemma \ref{tecnical}:

\begin{proof}[Proof of Lemma \ref{tecnical}]
Let $\mu\vdash4K$. By Proposition \ref{6_parts}, for every $n\geq3K$, $O\in\Omega_{n,n+K}$
and  $P_{n}\left(m_{1},...,m_{n}\right)=\text{Sub}_{\left(\text{x},\text{y}\right)}^{T}\left(P_{O}\right)$,
there are at most $16K^{2}$ monomials in $\left\{ m_{1},...,m_{n}\right\} $
which are not of the form $y_{1}$. Hence, for all $n$ large enough, one can write  $P_{n}\left(m_{1},...,m_{n}\right)=P^{\left(n-D\right)}\left(m_{1},...,m_{D}\right)$
for a fixed $D$ and monomials $m_{1},...,m_{D}$
with 
\begin{gather*}
2\leq\left|m_{i}\right|\stackrel{\ref{6_parts}}{\leq}K+1,\,\,\,\,\,i=1,...,D.
\end{gather*}

Applying Proposition \ref{NNM}, we can write
\begin{gather*}
P_{n}\left(m_{1},...,m_{n}\right)=P^{\left(n-D\right)}\left(m_{1},...,m_{D}\right)=\left(n-D\right)!\underset{\sigma\in S_{D}}{\sum}F_{n-D}\left(m_{\sigma\left(1\right)},...,m_{\sigma\left(D\right)}\right).
\end{gather*}

Now, by Proposition \ref{d_large}, for every $p>0$, if $n$ is large enough we have $\left|\text{C\ensuremath{\left(u\right)}}\right|>p$ for each monomial $u$
in $P_{n}\left(m_{1},...,m_{n}\right)$. 
Hence,  given that $n$ is sufficiently large, Lemma \ref{tec_oftec} implies that for any $\sigma\in S_{D}$ and monomial $u$ in $F_{n-D}\left(m_{\sigma\left(1\right)},...,m_{\sigma\left(D\right)}\right)$  the coefficient of $u^{\left(s\right)}$ in 
\begin{gather*}
F_{n+s-D}\left(m_{\sigma\left(1\right)},...,m_{\sigma\left(D\right)}\right)=\frac{1}{\left(n+s-D\right)!}P^{\left(n+s-D\right)}\left(m_{1},...,m_{n}\right)
\end{gather*}
 is a polynomial in $s$, and the lemma follows. 

\end{proof}

\section{Estimates on \texorpdfstring{$\text{dim}\left(W_{n,m}\right)$}{Lg} }

In this section, we give upper and lower bounds on $\text{dim}\left(W_{n,m}\right)$,
and review the special cases $m=n+1$ and $m=n+2$, where $\text{dim}\left(W_{n,m}\right)$
is precisely known. We also conjecture the asymptotic behavior of $\text{dim}\left(W_{n,n+K}\right)$ for a fixed $K$. It should be noted that for a given $n$, we are only interested in the case $m<n^{2}$, since we already know that 
$\text{dim}\left(W_{n,m}\right)=m!$ for every $n^{2}\leq m$ (see Proposition \ref{lst} and the paragraph following it).   

\subsection{Upper bound}

\begin{lemma}
\label{upper}
$\text{dim}\left(W_{n,m}\right)\leq\left(\begin{array}{c}
m\\
n
\end{array}\right)\cdot\frac{\left(m-1\right)!}{\left(n-1\right)!}$.

\end{lemma}
\begin{proof}

By Corollary \ref{Deacrip}, $W_{n,m}$ is spanned by $\left\{ P_{O}:O\in\Omega_{n,m}\right\} $,
where $\Omega_{n,m}$ is the set of ordered  partitions of
$\left[m\right]=\left\{ 1,...,m\right\} $ divided into $n$ parts, so it suffices to show
that \linebreak $\left|\Omega_{n,m}\right|=\left(\begin{array}{c}
m\\
n
\end{array}\right)\cdot\frac{\left(m-1\right)!}{\left(n-1\right)!}$.
Indeed, start by choosing a word $w=w_{1}\cdot\cdot\cdot w_{m}$ whose
letters are $1,...,m$ in some order, and then divide $w$ into $n$
parts.
The last division requires a choice
of $n-1$ "cut places"
out of the $m-1$ possibilities,  so all this procedure can be
made in $m!\left(\begin{array}{c}
m-1\\
n-1
\end{array}\right)$ ways. However, each such cut is counted $n!$ times (as we do not
care about their order), \linebreak implying that
\begin{gather*}
\left|\Omega_{n,m}\right|=\left(\begin{array}{c}
m-1\\
n-1
\end{array}\right)\cdot\frac{m!}{n!}=\left(\begin{array}{c}
m\\
n
\end{array}\right)\cdot\frac{\left(m-1\right)!}{\left(n-1\right)!}.
\end{gather*}

\end{proof}

\begin{remark}
As we will see in Theorem \ref{olsson},  $\text{dim}\left(W_{n,n+1}\right)=n\left(n+1\right)=\left(\begin{array}{c}
n+1\\
n
\end{array}\right)\cdot\frac{n!}{\left(n-1\right)!}$, so this bound is tight. 
\end{remark}

\subsection{Lower bound}

The following theorem is due to Latyshev \cite{latyshev1972regev} (see also {\cite[p.137]{kanel2015computational}}):

\begin{theorem}
\label{laty}

Let $f\in\mathbb{F}\left\langle X\right\rangle $ be a polynomial
of degree $n$, and set  $c_{m}:=\text{dim}\left(V_{m}/\left(V_{m}\cap\left(f\right)^{T}\right)\right)$.
Then 
\begin{gather*}
c_{m}\leq\underset{\underset{\ell\left(\lambda\right)<n}{\lambda\vdash m}}{\sum}\left(d^{\lambda}\right)^{2},
\end{gather*}

where $d^{\lambda}$ is the dimension of the irreducible representation
$S^{\lambda}$ of $S_{m}$. 

\end{theorem}

\begin{corollary}
For every $n,m\in\mathbb{N}$,$\underset{\underset{\ell\left(\lambda\right)\geq n}{\lambda\vdash m}}{\sum}d_{\lambda}^{2}\leq\text{dim}\left(W_{n,m}\right)$.
\end{corollary}
\begin{proof}
By Theorem \ref{laty}, $\text{dim}\left(V_{m}/\left(V_{m}\cap\left(x^{n}\right)^{T}\right)\right)=\text{dim}\left(V_{m}/W_{n,m}\right)\leq\underset{\underset{\ell\left(\lambda\right)<n}{\lambda\vdash m}}{\sum}\left(d^{\lambda}\right)^{2}$,
and thus
\begin{gather*}
m!-\text{dim}\left(W_{n,m}\right)=\text{dim}\left(V_{m}\right)-\text{dim}\left(W_{n,m}\right)\leq\underset{\underset{\ell\left(\lambda\right)<n}{\lambda\vdash m}}{\sum}\left(d^{\lambda}\right)^{2}.
\end{gather*}
 Using the well-known fact $m!=\underset{\lambda\vdash m}{\sum}\left(d^{\lambda}\right)^{2}$,
we obtain 
\begin{gather*}
\underset{\underset{\ell\left(\lambda\right)\geq n}{\lambda\vdash m}}{\sum}\left(d^{\lambda}\right)^{2}=m!-\underset{\underset{\ell\left(\lambda\right)<n}{\lambda\vdash m}}{\sum}\left(d^{\lambda}\right)^{2}\leq W_{n,m}.
\end{gather*}
\end{proof}
However, according to our observations, this bound seems
far from being optimal. The following is a better lower bound, in the special case when $n$ is coprime to $m$:

\begin{proposition}
If $n$ is coprime to $m$, then $\frac{m!}{\left(n-1\right)!}\leq\text{dim}\left(W_{n,m}\right)$.
\end{proposition}
\begin{proof}
We begin with some reductions.

\textbf{Reduction 1.} Consider the partition $\lambda=(m+1-n,\stackrel{n-1}{\overbrace{1,...,1}})\vdash m$,
and let 
\begin{gather*}
W_{n,m}^{\lambda}:=\text{span}_{\mathbb{F}}\left\{ P_{O}\mid O\in\Omega_{n,m}^{\lambda}\right\} \subseteq W_{n,m},
\end{gather*}
where  $\Omega_{n,m}^{\lambda}$ is the set of ordered partitions $O=\left\{ A_{1},...,A_{n}\right\} $
such that $\left|A_{n}\right|=m+1-n$ and $\left|A_{i}\right|=1$
for $i=1,...,n-1$ (see Definition \ref{parpar}). It suffices to show that 
\begin{align}
  \text{dim}_{\mathbb{F}}\left(W_{n,m}^{\lambda}\right)=\frac{m!}{\left(n-1\right)!}.
\label{R6}
\end{align}

\textbf{Reduction 2.} In order to prove (\ref{R6}), we  may assume that $\mathbb{F}=\mathbb{C}$. 

\textbf{Reduction 3.} It is not hard to see that the action of $S_{m}$
on $W_{n,m}^{\lambda}$ is transitive, and therefore $W_{n,m}^{\lambda}$
is a cyclic $\mathbb{\mathbb{C}}\left[S_{m}\right]$-module. In particular,
taking
\begin{gather*}
\Gamma=\left\{ [\![1]\!],[\![2]\!],...,[\![n-1]\!],[\![n....m]\!]\right\} \in\Omega_{n,m}^{\lambda}
\end{gather*}
then
\begin{gather*}
W_{n,m}^{\lambda}=\mathbb{\mathbb{C}}\left[S_{m}\right]\cdot P_{n}\left(x_{1},x_{2},...,x_{n-1},x_{n}\cdot\cdot\cdot x_{m}\right)=\mathbb{\mathbb{C}}\left[S_{m}\right]\cdot P_{\Gamma}.
\end{gather*}

Identifying $\mathbb{\mathbb{C}}\left[S_{m}\right]=V_{m}$ as usual,
one may observe that 
\begin{flalign*}
P_{\Gamma} & =P_{n}\left(x_{1},x_{2},...,x_{n-1},x_{n}\cdot\cdot\cdot x_{m}\right)\\
 & =\left(\underset{\sigma\in {\color{olive}S_{n-1}}}{\sum}\sigma\right)\left({\color{olive}x_{1}\cdot\cdot\cdot x_{n-1}}x_{n}\cdot\cdot\cdot x_{m}+{\color{olive}x_{2}\cdot\cdot\cdot x_{n-1}}x_{n}\cdot\cdot\cdot x_{m}{\color{olive}x_{1}}+...+x_{n}\cdot\cdot\cdot x_{m}{\color{olive}x_{1}\cdot\cdot\cdot x_{n-1}}\right)\\
 & =\left(\underset{\sigma\in S_{n-1}}{\sum}\sigma\right)\left(e+g+...+g^{n-1}\right),
\end{flalign*}

where $g$ is the permutation $g=\left(1,2,...,m\right)\in S_{m}$.
Thus, if we set
$S_{n-1}^{+}:=\underset{\sigma\in S_{n-1}}{\sum}\sigma$ and \linebreak $\zeta_{n-1}:=e+g+g^{2}+...+g^{n-1}$, it suffices to show that
\begin{align}
\text{dim}\left(\mathbb{\mathbb{C}}\left[S_{m}\right]\cdot P_{\Gamma}\right)=\text{dim}\left(\mathbb{\mathbb{C}}\left[S_{m}\right]\cdot S_{n-1}^{+}\cdot\zeta_{n-1}\right)=\frac{m!}{\left(n-1\right)!}.
\label{R7}
\end{align}

To prove (\ref{R7}), we need the following lemma (which as far as the author is aware, does not appear in the literature).

\begin{lemma}
\label{inertible}
Let $G=\left\langle g\right\rangle $ be the cyclic
group of order $n$, and let $0\leq k<n$. The element 
$\zeta_{k}:=e+g+...+g^{k}$
is invertible in $\mathbb{C}\left[G\right]$ if and only if $k+1$ is coprime to $n$.
\end{lemma}
The proof of this lemma can be found in the
the \hyperlink{Appendix}{Appendix}.
However, assuming this lemma, the rest of the proof is straightforward: since we assume that $n$ is coprime to $m$, we obtain that $\zeta_{n-1}$ is invertible in $\mathbb{\mathbb{C}}\left[\left\langle g\right\rangle \right]\subseteq\mathbb{\mathbb{C}}\left[S_{m}\right]$,
and thus one can deduce that 
\begin{gather*}
\text{dim}\left(\mathbb{\mathbb{C}}\left[S_{m}\right]\cdot S_{n-1}^{+}\cdot\zeta_{n-1}\right)=\text{dim}\left(\mathbb{\mathbb{C}}\left[S_{m}\right]\cdot S_{n-1}^{+}\right).
\end{gather*}
Taking a left transversal $\left\{ \tau_{1},...,\tau_{\left[S_{m}:S_{n-1}\right]}\right\} $
for $S_{n-1}$ in $S_{m}$, then $\left\{ \tau_{1}\cdot S_{n-1}^{+},....,\tau_{\left[S_{m}:S_{n-1}\right]}\cdot S_{n-1}^{+}\right\} $
is a basis of $\mathbb{\mathbb{C}}\left[S_{m}\right]\cdot S_{n-1}^{+}$ over $\mathbb{C}$, as one can easily verify. 
It follows that 
\begin{gather*}
\text{dim}\left(\mathbb{\mathbb{C}}\left[S_{m}\right]\cdot S_{n-1}^{+}\cdot\zeta_{n-1}\right)=\text{dim}\left(\mathbb{\mathbb{C}}\left[S_{m}\right]\cdot S_{n-1}^{+}\right)=\left[S_{m}:S_{n-1}\right]=\frac{m!}{\left(n-1\right)!}. \end{gather*}
\end{proof}
\begin{remark}
\label{forr}
In fact, this proof can be implemented to show that there is an isomorphism
of $S_{m}$-representations $W_{n,m}^{\lambda}\cong\text{Ind}_{S_{n-1}}^{S_{m}}\left(\mathbf{1}_{S_{n-1}}\right)$, where $\mathbf{1}_{S_{n-1}}$ is the trivial representation of $S_{n-1}$.
 We leave the
details to the reader.
\end{remark}
\subsection{ \texorpdfstring{$\text{dim}\left(W_{n,n+K}\right)$ when $K$}{Lg} is fixed}
When $K$ is fixed, the dimension of $W_{n,n+K}$ and its decomposition into irreducible representations are known only in two cases: $K=1$, due to Olsson and Regev, and $K=2$, due to Drensky and Benanti.
\subsubsection{ \texorpdfstring{$W_{n,n+1}$ and  $W_{n,n+2}$}{Lg}  }
\begin{theorem}[Olsson and Regev \cite{olsson1977thet}]
\label{olsson}
 For every $n\geq3$, the following isomorphism
of \linebreak $S_{n+1}$-representations holds: 
\begin{gather*}
W_{n,n+1}=S^{\left(n+1\right)}\oplus2S^{\left(n,1\right)}\oplus S^{\left(n-1,2\right)}\oplus S^{\left(n-1,1^{2}\right)},
\end{gather*}
and $\text{dim}\left(W_{n,n+1}\right)=n\left(n+1\right)$.
\end{theorem}

\begin{remark}
Observing that $\Omega_{n,n+1}=\Omega_{n,n+1}^{(2,1^{n-1})}$, Theorem \ref{olsson} follows directly from Equation (\ref{R6}) and Remark \ref{forr}. Indeed, $n$ is coprime to $n+1$, and it is well known (see, e.g., {\cite[p.148]{diaconis1988group}}) that  $\text{Ind}_{S_{n-1}}^{S_{n+1}}\left(\mathbf{1}_{s_{n-1}}\right)$ is isomorphic to $S^{\left(n+1\right)}\oplus2S^{\left(n,1\right)}\oplus S^{\left(n-1,2\right)}\oplus S^{\left(n-1,1^{2}\right)}$.

\end{remark}

\begin{theorem}[Drensky and Benanti \cite{benanti1999polynomial}]
\label{banana}

For every $n\geq6$, the following isomorphism
of \linebreak $S_{n+2}$-representations holds: 

\begin{flalign*}
W_{n,n+2} & =S^{\left(n+2\right)}\oplus4S^{\left(n+1,1\right)}\oplus5S^{\left(n,2\right)}\oplus5S^{\left(n,1^{2}\right)}\oplus3S^{\left(n-1,3\right)}\oplus6S^{\left(n-1,2,1\right)}\\
 & \oplus3S^{\left(n-1,1^{3}\right)}\oplus S^{\left(n-2,4\right)}\oplus S^{\left(n-2,3,1\right)}\oplus2S^{\left(n-2,2^{2}\right)}\oplus S^{\left(n-2,2,1^{2}\right)}\oplus S^{\left(n-2,1^{4}\right)}.
\end{flalign*}

\end{theorem}

\begin{corollary}
\label{colcol}
For every $n\geq6$, $\text{dim}W_{n,n+2}=\frac{1}{2}n\left(n+2\right)\left(n^{2}+2n-1\right)$.
\end{corollary}

\begin{proof}
This is a straightforward calculation using  the hook length formula (Theorem \ref{hooh}). We leave the
details to the reader.
\end{proof}
\subsubsection{ \texorpdfstring{$W_{n,n+K}$ for  $K>2$}{Lg}  }
When $K$ is fixed, the upper bound of Theorem \ref{upper} yields
\begin{flalign*}
\text{dim}\left(W_{n,n+K}\right) & \leq\left(\begin{array}{c}
n+K\\
n
\end{array}\right)\cdot\frac{\left(n+K-1\right)!}{\left(n-1\right)!}=\frac{1}{K!}\left(n+K\right)\cdot\left(n+K-1\right)^{2}\cdot\cdot\cdot\left(n+1\right)^{2}\cdot n.
\end{flalign*}
This bound, combined with the above results, suggests the following conjecture: 
\begin{conjecture}
Fix $K\in\mathbb{N}$. There exist constants $c_{K},C_{K}>0$ such
that 
\begin{gather*}
c_{K}\cdot n^{2K}\leq\text{dim}\left(W_{n,n+K}\right)\leq C_{K}\cdot n^{2K}.
\end{gather*}
\end{conjecture}

Next, we exploit our main result to show that the dimension of $W_{n,n+K}$ can be described (for a fixed $K$) by a polynomial, given that $n$ is sufficiently large.
\begin{proposition}
\label{polyp}
Given $D\in\mathbb{N}$ and a partition $\lambda\vdash D$, there
exists a polynomial $q\left(x\right)$ such that $\text{dim}\left(S^{\lambda^{\left(d\right)}}\right)=q\left(d\right)$
for every $d\in\mathbb{N}\cup\left\{ 0\right\} $.
\end{proposition}

 \begin{proof}
Let $\lambda=\left(\lambda_{1},...,\lambda_{p}\right)\vdash D$.  Writing
the hook lengths $\left\{ h_{i,j}^{\lambda}\right\} $ and
$\left\{ h_{i,j}^{\lambda^{\left(d\right)}}\right\} $ inside $D_{\lambda}$
and $D_{\lambda^{\left(d\right)}}$ (see Definition \ref{hookk} and Example
\ref{ex_will}), we obtain
\begin{center}
\begin{figure}[H]
    \centering
    \includegraphics[width=1\textwidth]{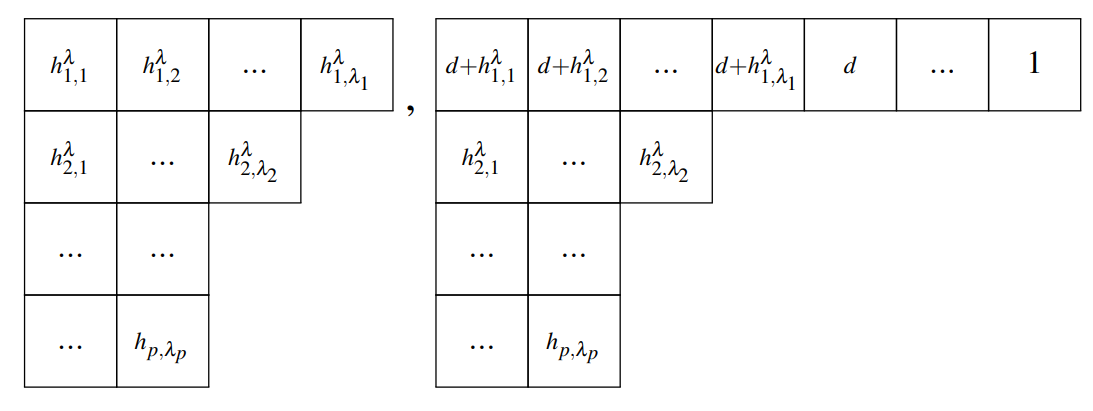}.
\end{figure}
\end{center}

Thus, by the hook length formula (Theorem \ref{hooh}), \begin{flalign*}
\text{dim}\left(S^{\lambda^{\left(d\right)}}\right) & =\frac{\left(d+D\right)!}{\underset{i\neq1,\left(i,j\right)\in D_{\lambda}}{\prod}h_{i,j}^{\lambda}\cdot\left(d+h_{1,1}^{\lambda}\right)\cdot\cdot\cdot\left(d+h_{1,\lambda_{1}}^{\lambda}\right)\cdot d!}\\
 & =\frac{\left(d+D\right)\cdot\cdot\cdot\left(d+2\right)\left(d+1\right)}{\underset{i\neq1,\left(i,j\right)\in D_{\lambda}}{\prod}h_{i,j}^{\lambda}\cdot\left(d+h_{1,1}^{\lambda}\right)\cdot\cdot\cdot\left(d+h_{1,\lambda_{1}}^{\lambda}\right)}
\end{flalign*}
which is a polynomial in $d$.
\end{proof}

\begin{corollary}
\label{popo}
Fix $K\in\mathbb{N}$. There is an integer $N_{K}>0$ and a polynomial $p_{K}\left(x\right)\in\mathbb{F}\left[x\right]$ such that, for every $n\geq N_{K}$, $p_{K}\left(n\right)=\text{dim}\left(W_{n,n+K}\right)$.
\end{corollary}

\begin{proof}
We have seen in Theorem \ref{stab}   that the decomposition of $W_{n,n+K}$ is
stabilized,
where the decomposition involves only partitions derived from partitions of
$4K$. By Proposition \ref{polyp}, for each $\lambda\vdash4K$, the dimensions
$\left(\text{dim}\left(S^{\lambda^{\left(d\right)}}\right)\right)_{d\in\mathbb{N}}$
can be described by a polynomial. Hence, if $n$ large enough, 
\begin{gather*}
\text{dim}\left(W_{n,n+K}\right)=\underset{\lambda\vdash4K}{\sum}m^{\lambda}\text{dim}\left(S^{\lambda^{\left(n-3K\right)}}\right)
\end{gather*}
can also be described by a polynomial.
\end{proof}
\begin{example}
According to Theorem \ref{olsson} and Corollary \ref{colcol}, $p_{1}\left(n\right)=n\left(n+1\right)$
for every $n\geq3$, and $p_{2}\left(n\right)=\frac{1}{2}n\left(n+2\right)\left(n^{2}+2n-1\right)$ for every $n\geq6$.
\end{example}
\begin{conjecture}
In Corollary \ref{popo}, one can take $N_{K}=2K$.
\end{conjecture}
\begin{problem}
 \textit{Find $p_{3}\left(n\right)$, and calculate the decomposition
of $W_{n,n+3}$}.
\end{problem}

\section*{\hypertarget{Appendix}{\textbf{Appendix - proof of Lemma \ref{inertible}}}}

\begin{lemma*}
Let $G=\left\langle g\right\rangle $ be the cyclic
group of order $n$, and let $0\leq k<n$. The element \linebreak
$\zeta_{k}:=e+g+...+g^{k}$
is invertible in $\mathbb{C}\left[G\right]$ if and only if $k+1$ is coprime to $n$.
\end{lemma*}

\begin{proof}
Note that $\zeta_{k}$ is invertible in $\mathbb{C}\left[G\right]$
if and only if it is not contained in any proper ideal. Since proper
ideals in $\mathbb{C}\left[G\right]$ are direct sums of one-dimensional
irreducible representations, $\zeta_{k}$ being irreducible is equivalent to the
following: for every irreducible representation $V$ of $G$, and
for every idempotent $e_{V}$ corresponding to $V$ (which is just the projection on $V$), 
$\zeta_{k}e_{V}\neq0$. Clearly, we always may assume that $e_{V}\neq\text{Id}$.

Set $\omega_{n}:=\text{exp}\left(\frac{2\pi i}{n}\right)$, and recall
that the characters of irreducible representations of $G$ are given
in $\left\{ \chi_{0},\chi_{1},...,\chi_{n-1}\right\}$, where $\chi_{l}\left(g^{s}\right)=\omega_{n}^{sl}$, and so the corresponding
idempotents $\left\{ e_{\chi_{0}},e_{\chi_{1}},...,e_{\chi_{n-1}}\right\} $
are given by the well-known formula
\begin{gather*}
e_{\chi_{l}}=\frac{1}{n}\stackrel[s=0]{n-1}{\sum}\chi_{l}\left(g^{-s}\right)g^{s}=\frac{1}{n}\stackrel[s=0]{n-1}{\sum}\omega_{n}^{-sl}g^{s}.
\end{gather*}

Suppose $\zeta_{k}e_{\chi_{l}}=0$ for some $l\neq0$ (note that $e_{\chi_{0}}=Id$). Since 
\begin{gather*}
ge_{\chi_{l}}=\frac{1}{n}\stackrel[s=0]{n-1}{\sum}\omega_{n}^{-sl}g^{s+1}=\omega_{n}^{l}e_{\chi_{l}},
\end{gather*}
it follows that
\begin{align}
0=\zeta_{k}e_{\chi_{l}}=\left(1+\omega_{n}^{l}+...+\omega_{n}^{kl}\right)e_{\chi_{l}}=\left(\frac{1-\left(\omega_{n}^{l}\right)^{k+1}}{1-\omega_{n}^{l}}\right)e_{\chi_{l}},
\label{t0}
\end{align}

which implies $\left(\omega_{n}^{l}\right)^{k+1}=1$. In particular
\begin{gather*}
n\mid l\left(k+1\right),
\end{gather*}
and (as $l<n$) $k+1$ and $n$ are not coprime.

Conversely, if $k+1$ and $n$ are not coprime, there exists
$l<n$ such that $n\mid l\left(k+1\right)$, and thus by Equation (\ref{t0}) we obtain  $\zeta_{k}e_{\chi_{l}}=0$, so $\zeta_{k}$ is not an invertible
element.
\end{proof}


\begin{thebibliography}{10}

\bibitem{benanti1999polynomial}
F.~Benanti and V.~Drensky.
\newblock Polynomial identities of nil algebras of bounded index.
\newblock {\em Bollettino dell'Unione Matematica Italiana}, 2:673--691, 1999.

\bibitem{berele1982homogeneous}
A.~Berele.
\newblock Homogeneous polynomial identities.
\newblock {\em Israel journal of mathematics}, 42(3):258--272, 1982.

\bibitem{dent2000conjecture}
S.C. Dent and J.~Siemons.
\newblock On a conjecture of {Foulkes}.
\newblock {\em Journal of Algebra}, 226(1):236--249, 2000.

\bibitem{diaconis1988group}
P.~Diaconis.
\newblock {Group Representations in Probability and Statistics}.
\newblock {\em Lecture notes-monograph series}, 11:i--192, 1988.

\bibitem{drensky1984codimensions}
V.~Drensky.
\newblock {Codimensions of {$T$-ideals} and Hilbert series of relatively free
  algebras}.
\newblock {\em Journal of Algebra}, 91(1):1--17, 1984.

\bibitem{drensky2012polynomial}
V.~Drensky and E.~Formanek.
\newblock {\em {Polynomial Identity Rings}}.
\newblock Birkh{\"a}user, 2012.

\bibitem{dubnov1943abaissement}
J.~Dubnov and V.~Ivanov.
\newblock Sur l’abaissement du degr{\'e} des polyn{\^o}mes en affineurs.
\newblock In {\em Dokl. Akad. Nauk SSSR}, volume~41, pages 95--98, 1943.

\bibitem{giambruno2005polynomial}
A.~Giambruno and M.~Zaicev.
\newblock {\em {Polynomial Identities and Asymptotic Methods}}.
\newblock Number 122. American Mathematical Soc., 2005.

\bibitem{henke235explicit}
A.~Henke and A.~Regev.
\newblock Explicit decompositions of the group algebras
  {$\mathbb{F}\left[S_{n}\right]$ and $\mathbb{F}\left[A_{n}\right]$}.
\newblock {\em Lecture Notes in Pure and Appl. Math}, 235:329--357.

\bibitem{higman1956conjecture}
G.~Higman and P.~Hall.
\newblock On a conjecture of {Nagata}.
\newblock In {\em Mathematical Proceedings of the Cambridge Philosophical
  Society}, volume~52, pages 1--4. Cambridge University Press, 1956.

\bibitem{james2006representation}
G.D. James.
\newblock {\em {The Representation Theory of the Symmetric Groups}}, volume
  682.
\newblock Springer, 2006.

\bibitem{kanel2015computational}
A.~Kanel-Belov, Y.~Karasik, and L.H. Rowen.
\newblock {\em Computational Aspects of Polynomial Identities: Volume l,
  Kemer's Theorems}, volume~16.
\newblock CRC Press, 2015.

\bibitem{kuzmin1975nagata}
E.N. Kuzmin.
\newblock On the {Nagata-Higman} theorem.
\newblock {\em Proceedings dedicated to the 60th birthday of Academician Iliev,
  Sofia}, pages 101--107, 1975.

\bibitem{latyshev1972regev}
V.N. Latyshev.
\newblock On {Regev's} theorem on identities in a tensor product of
  {PI-algebras}.
\newblock {\em Uspekhi Matematicheskikh Nauk}, 27(4):213--214, 1972.

\bibitem{10.2969/jmsj/00430296}
M.~Nagata.
\newblock {On the nilpotency of nil-algebras.}
\newblock {\em Journal of the Mathematical Society of Japan}, 4(3-4):296 --
  301, 1952.

\bibitem{olsson1977thet}
J.B. Olsson and A.~Regev.
\newblock On the {$T$-ideal} generated by a standard identity.
\newblock {\em Israel Journal of Mathematics}, 26(2):97--104, 1977.

\bibitem{procesi2007lie}
C.~Procesi.
\newblock {\em Lie groups: an approach through invariants and representations},
  volume 115.
\newblock Springer, 2007.

\bibitem{razmyslov1974trace}
Yu.~P. Razmyslov.
\newblock Trace identities of full matrix algebras over a field of
  characteristic zero.
\newblock {\em Mathematics of the USSR-Izvestiya}, 8(4):727, 1974.

\bibitem{regev1980polynomial}
A.~Regev.
\newblock The polynomial identities of matrices in characteristic zero.
\newblock {\em Communications in {Algebra}}, 8(15):1417--1467, 1980.

\bibitem{regev1982polynomial}
A.~Regev.
\newblock A polynomial rate of growth for the multiplicities in cocharacters of
  matrices.
\newblock {\em Israel Journal of Mathematics}, 42(1):65--74, 1982.

\bibitem{rowen2008graduate}
L.H. Rowen.
\newblock {\em Graduate Algebra: Noncommutative View: Noncommutative View},
  volume~2.
\newblock American Mathematical Soc., 2008.

\bibitem{vaughan1993algorithm}
M.~Vaughan-Lee.
\newblock An algorithm for computing graded algebras.
\newblock {\em Journal of {Symbolic Computation}}, 16(4):345--354, 1993.

\end{thebibliography}


\newpage
\end{document}